\documentclass[a4paper]{amsart}

\usepackage{amsfonts}
\usepackage{amsmath}
\usepackage{amsthm}
\usepackage{booktabs}

\usepackage{hyperref}
\usepackage{xspace}
\usepackage[utf8]{inputenc}
\usepackage{color}
\usepackage{mathtools}
\usepackage{esvect}
\usepackage{enumerate}
\usepackage{xcolor}
\usepackage{dsfont}

\usepackage{commath}

\usepackage{yhmath}

\usepackage{bm}
\usepackage{comment}
\usepackage{cleveref}

\definecolor{aliceblue}{rgb}{0.94, 0.97, 1.0}
\definecolor{apricot}{rgb}{0.98, 0.81, 0.69}
\definecolor{antiquewhite}{rgb}{0.98, 0.92, 0.84}
\usepackage{xargs}
\usepackage[colorinlistoftodos,prependcaption,textsize=scriptsize,linecolor=orange,backgroundcolor=antiquewhite,bordercolor=orange]{todonotes}

\newcommandx{\info}[2][1=]{\todo[linecolor=blue,backgroundcolor=aliceblue,bordercolor=blue,#1]{#2}}

\usepackage[style=numeric, backend=bibtex, url=true]{biblatex}
\hypersetup{breaklinks=true}
\allowdisplaybreaks

\definecolor{dgreen}{rgb}{0.,0.45,0.}
\definecolor{dblue}{rgb}{0.,0.,0.55}
\newcommand{\dgt}[1]{\textcolor{black}{#1}}
\newcommand{\dbt}[1]{\textcolor{black}{#1}}

\newcounter{counter}
\theoremstyle{plain}
\newtheorem{thm}[counter]{Theorem}
\newtheorem{lem}[counter]{Lemma}
\newtheorem{cor}[counter]{Corollary}
\newtheorem{prop}[counter]{Proposition}
\theoremstyle{definition}
\newtheorem{example}[counter]{Example}
\newtheorem{defn}[counter]{Definition}
\newtheorem{rem}[counter]{Remark}

\newcommand{\mull}{\textbf{m}}

\newcommand{\diag}{\text{diag}}

\newcommand{\loc}{\text{loc}}

\newcommand{\corr}{\text{corr}}
\newcommand{\cov}{\text{Cov}}
\newcommand{\var}{\text{Var}}

\makeatletter
\newcommand{\amsvect}{%
  \mathpalette {\overarrow@\vectfill@}}
\def\vectfill@{\arrowfill@\relbar\relbar{\raisebox{-3.81pt}[\p@][\p@]{$\mathord\mathchar"017E$}}}
\newcommand{\amsvectb}{%
  \mathpalette {\overarrow@\vectfillb@}}
  \newcommand{\vecbar}{%
  \scalebox{0.8}{$\relbar$}}
\def\vectfillb@{\arrowfill@\vecbar\vecbar{\raisebox{-4.35pt}[\p@][\p@]{$\mathord\mathchar"017E$}}}
\makeatother

\newcommand{\fak}{\mathbf{k}}

\makeatletter
\newcommand\definealphabetloop[3]{%
  \ifx\relax#3\expandafter\@gobble\else\expandafter\@firstofone\fi
  {\expandafter\providecommand\expandafter*\csname#1#3\endcsname{#2{#3}}%
   \definealphabetloop{#1}{#2}}%
}%
\newcommand\definealphabet[2]{%
  \definealphabetloop{#1}{#2}abcdefghijklmnopqrstuvwxyzABCDEFGHIJKLMNOPQRSTUVWXYZ\relax
}%
\definealphabet{b}{\mathbb}%
\definealphabet{c}{\mathcal}%
\definealphabet{f}{\mathbf}%
\definealphabet{r}{\mathrm}%
\definealphabet{s}{\mathscr}%
\definealphabet{t}{\mathtt}%
\makeatother

\renewcommand{\set}[1]{\left\{#1\right\}}
\renewcommand{\norm}[1]{\left\|#1\right\|}

\usepackage{accents}
\newlength{\dhatheight}
\newcommand{\doublehat}[1]{%
    \settoheight{\dhatheight}{\ensuremath{\hat{#1}}}%
    \addtolength{\dhatheight}{-0.35ex}%
    \hat{\vphantom{\rule{1pt}{\dhatheight}}%
    \smash{\hat{#1}}}}

\begin{document}

\title{Moments of Generalized Fractional Polynomial Processes}
\date{\today}

\author{Johannes Assefa}
\address{Department of Mathematical Stochastics\\Dresden Technical University}
\email{johannes.assefa@tu-dresden.de}
\author{Martin Keller-Ressel}
\address{Department of Mathematical Stochastics\\Dresden Technical University}
\email{martin.keller-ressel@tu-dresden.de}


\begin{abstract}
We derive a moment formula for generalized fractional polynomial processes, i.e., for polynomial-preserving Markov processes time-changed by an inverse L\'evy-subordinator. 
If the time change is inverse $\alpha$-stable, the time-derivative of the Kolmogorov backward equation is replaced by a Caputo fractional derivative of order $\alpha$, and we demonstrate that moments of such processes are computable, in a closed form, using matrix Mittag-Leffler functions. The same holds true for cross-moments in equilibrium, generalizing results of Leonenko, Meerschaert and Sikorskii from the one-dimensional diffusive case of second-order moments to the multivariate, jump-diffusive case of moments of arbitrary order. We show that also in this more general setting, fractional polynomial processes exhibit long-range dependence, with correlations decaying as a power law with exponent $\alpha$.
\end{abstract}

\maketitle

\tableofcontents

\section{Introduction}
   Polynomial processes, which have been extensively studied in both finite and infinite-dimensional settings \cite{Cuchiero12,Filipovic16, Filipovic17, Filipovic19,Cuchiero21,Guida21}, are particularly notable for their tractable moment calculations. Specifically, their conditional moments can be expressed through matrix exponentials, rendering moment calculations for these processes computationally efficient. \dgt{This makes them particularly well-suited to model problems in finance such as asset pricing and equilibrium analysis \cite{Filipovic17}, many of which are naturally described by jump diffusion models that fall within the framework of polynomial processes \cite{Cuchiero12}}. \par 
\dgt{Jump diffusion processes can be derived as scaling limits of continuous-time random walks with exponentially distributed waiting times between successive jumps. If the waiting times are general random times, the corresponding scaling limit is a time-changed jump diffusion, often referred to as anomalous diffusion \cite{Nane11,Meerschaert13,Orsingher18,Leonenko13}.} 
Time-changed jump diffusions are goverened by generalized time-fractional equations and hence termed generalized fractional jump-diffusions. In particular, for a Markov process $(X_t)_{t\geq0}$ 
with infinitesimal generator $(\cG,\cD(\cG))$ and a subordinator $\sigma^f=(\sigma^f_s)_{s\geq0}$ with Laplace exponent $f$ and hitting time process $L_t$, $t\geq0$, the average behaviour of the time-changed process, $(t,x)\mapsto\bE_x[u(X_{L_t})]$, is governed by time-fractional equations of the form 
\begin{align*}
	\begin{cases}
		\bD^f_tq(t,x)=\cG q(t,x),\quad0<t<\infty,\\
		 q(0,\cdot)=u\in\cD(\cG),
	\end{cases}
\end{align*}
where $\bD^f_t$ is the generalized $f$-Caputo fractional derivative defined in (\ref{def:caputo derivative}), 
introduced in \cite{Chen17}. 
In the special case where $\sigma^f$ is $\alpha$-stable with $f(\lambda)=\lambda^\alpha$ and $\alpha\in(0,1)$, one obtains the (conventional) Caputo fractional derivative 
\begin{align*}
		\bD^f_tq(t)
		=\frac{1}{\Gamma(1-\alpha)}\frac{d}{dt}\int_0^t(t-s)^{-\alpha}(q(s)-q(0))\,ds,\quad0<t<\infty,
	\end{align*}
and the time-changed process $(X_{L_t})_{t\geq0}$ is known as a fractional jump-diffusion.
Generalized fractional jump-diffusions are particularly useful for capturing periods of motionlessness interspersed with diffusive periods with applications spanning from biology to finance, see \cite{Sikorskii12} for full details. \par 
\dbt{Here, we consider the case when the original process $(X_t)_{t\geq0}$ is a polynomial process.} Existing literature on polynomial processes relies heavily on their Markov property or on their characterization as Ito semimartingales \cite{Cuchiero12, Filipovic16, Filipovic17}. However, when polynomial processes are subjected to an inverse subordinator, these properties are lost. 
Hence, it is unclear, whether the tractable moment formulas for polynomial processes in the Markov setting can be transferred to the non-Markovian setting of inverse subordination. In this paper, we address this gap, by demonstrating that these inverse time-changed polynomial processes retain their polynomial-preserving structure, in the following sense: in \Cref{sec:inverse subordinated polynomial processes} we show that 
the moments of generalized fractional polynomial processes can be determined as solutions to a \dgt{finite-dimensional} linear fractional differential equations associated with the underlying process, which generally takes the form
   \begin{align*}
	\begin{cases}
		\frac{d}{dt}\left(bq(t)+\int_0^tw(t-s)(q(s)-q(0))\,ds\right)=Aq(t),\quad 0<t<\infty,\\
		q(0)=u\in\bR^N,
	\end{cases}
   \end{align*}
   where $b\geq0$, $A\in\bR^{N\times N}$, $q(t)\in\bR^N$ for each $t\geq0$, and, $t\mapsto w(t)$ is a non-negative decreasing function on $(0,\infty)$ with a singularity at $t=0$. Specifically, if the subordinator is $\alpha$-stable, $\alpha\in(0,1)$, moments of the corresponding fractional polynomial process are obtained as solutions to 
    \begin{align*}
    	\begin{cases}
		\frac{1}{\Gamma(1-\alpha)}\frac{d}{dt}\int_0^t(t-s)^{-\alpha}(q(s)-q(0))\,ds=Aq(t),\quad 0<t<\infty,\\
		q(0)=u\in\bR^N,
	\end{cases}
    \end{align*}
    where $\Gamma(\lambda)\coloneqq \int_0^\infty t^{\lambda-1}e^{-t}\,dt$ is the Gamma function. In this case, in \Cref{sec:moment formula}, analytical solutions are provided in terms of Mittag-Leffler functions with matrix arguments, which allow to derive  easy to calculate expressions for 
    \begin{align*}
    	(t,x)\mapsto\bE_x[p(X_{L_t})],
    \end{align*}
    where $p$ is a polynomial, $X$ is a polynomial process, and $L_t$, $t\geq0$ is the inverse $\alpha$-stable subordinator.
    In \Cref{sec:cross moments}, we provide a cross-moment formula in terms of Mittag-Leffler functions for fractional polynomial processes in equilibrium. In particular, we derive analytically tractable expressions for
    \begin{align*}
    	(r,s)\mapsto\bE_\mu[p(X_{L_{r+s}})q(X_{L_r})],
    \end{align*}
    where $p$ and $q$ are polynomials, $X$ is a polynomial process, and $\mu$ is an \emph{$m$-limiting} distribution ($m\in\bN$), a weaker notion of equilibrium that describes the limiting moments of $X$ up to order $m$, without requiring convergence to a unique limiting distribution.    
     Among other applications, this facilitates the calculation of the (auto\nobreakdash)\-correlation function of the respective process as demonstrated in \Cref{sec:correlation}.   \par 
   We end this introduction with some conventions that will be used throughout this paper. 
   In what follows, $S$ is a closed subset of $\bR^d$ and $\cS$ denotes its Borel $\sigma$-algebra. \dgt{By $\cP_n(S)$ we denote the finite-dimensional vector space of polynomials up to degree $n\geq0$ on $S\subseteq\bR^d$ defined by
\[
\cP_n(S)=\left\{S\ni x\mapsto\sum_{\lvert k\rvert=0}^n\alpha_\fak x^\mathbf{k}\bigg|\alpha_\fak\in\bR\right\},
\]
where we use multi-index notation $\fak=(k_1,\dots,k_d)\in\bN^d_0$, $\lvert\fak\rvert=k_1+\cdots k_d$, and $x^\fak=x_1^{k_1}\cdots x_d^{k_d}$.
We set $N=\dim\cP_n(S)$ and note that in general $N$ could depend on $S$: if $S=\{x\}$ for some $x\in\bR^d$ then $N=1$ and for $S=\bR^d$ a simple combinatorial argument shows that $N=\binom{d+n}{n}$.
In the further course of this work we assume that $S$ has nonempty interior. In this case $\cP_n(S)$ can be identified with the restriction of polynomials in $\cP_n(\bR^d)$ to the set $S$ and we write $\cP_n\coloneqq\cP_n(S)$.} Then each $p\in\cP_n$ has a representation 
	\begin{align*}
		p(x)=H(x)^\top\vec{p},\quad x\in S,
	\end{align*}
	where $\vec{p}\in\bR^N$ and $H(x)^\top=(h_1(x),\dots h_N(x))$ denotes a basis of $\cP_n$, e.g. the monomial basis in graded lexicographic ordering \cite{Cox15}.
	\dgt{Given $q \in \cP_m$, we introduce the \emph{multiplication mapping}
\begin{align}\label{def:mul}
	\mull_q \colon\cP_n\to \cP_{n+m}, \quad p \mapsto pq.
\end{align}
As a linear mapping, $\mull_q$ has a matrix representation in a basis of $\cP_{n+m}$ which we denote by $M_q$. For example, for $q(x,y) = x + 2y - 1$ and $n=m=1$, the basis representation of $M_q$ in $H(x,y)=(1, x, y, x^2, xy, y^2)$ is given by
\begin{align}\label{eq:multiplication map}
M_q = \begin{pmatrix}-1 & 1 & 2 & 0 & 0 & 0\\ 0 & -1 & 0 & 1 & 2 &  0 \\ 0 & 0 & -1 & 0 & 1 & 2 \end{pmatrix}^\top.
\end{align}}

\section{Preliminaries}
In this section we collect some technical information which will be used throughout the paper.

\subsection{Functions and Laplace transforms evaluated at matrices}\label{sec:matrix}
\begin{defn} Let $A\in\bR^{n\times n}$ be a quadratic matrix with eigenvalues $\xi_1,\dots,\xi_m$. 
\begin{enumerate}[(i)]
\item We call 
\begin{align*}\label{eq:pi}
	\pi(A) = \max\{\Re\xi_i\colon \xi_1,\dots\xi_m\}
\end{align*}
the \emph{index of stability} of $A$. 
\item We say that $A$ is zero-stable, if $\pi(A) = 0$, $\xi = 0$ is the only eigenvalue with $\Re\,\xi = 0$, and $\xi = 0$ is a simple eigenvalue. 
\end{enumerate}
\end{defn}
From standard results on stability of linear ODEs \cite[Ch.~1.9]{Perko13} we have the following:
\begin{lem}\label{lem:A_estimate}Let $A\in\bR^{n\times n}$ be a quadratic matrix. 
\begin{enumerate}[(i)]
\item For any $\epsilon > 0$ there exists $M_\epsilon > 0$ such that $\norm{e^{At}} \le M_\epsilon e^{t(\pi(A) + \epsilon)}$ for all $t \ge 0$.
\item If $A$ is zero-stable, then there exists $M> 0$ such that $\norm{e^{At}} \le M$ for all $t \ge 0$.
\end{enumerate}
\end{lem}

Let $f: \bC \to \bR$ be a real function. Let $A = Q^{-1} J Q$ be the Jordan normal form of $A$ and let $J_k(\xi)$ be a Jordan block associated to eigenvalue $\xi$. Following \cite{Higham08}, the function $f$ can be extended to the matrix argument $A$ by setting 
\[f(A) = Q f(J) Q^{-1} = Q \diag(f(J_1(\xi_1)), \dotsc, f(J_l(\xi_m))) Q^{-1},\qquad l \le m\]
where
\begin{equation}\label{eq:f_jordan}
f(J_k(\xi)) = \begin{pmatrix} 
f(\xi) & f'(\xi) & \dotsm & \frac{f^{(m_k - 1)}(\xi)}{(m_k - 1)!} \\
 & f(\xi) & \ddots & \vdots \\
 & & \ddots & f'(\xi)\\
 & & & f(\xi)
\end{pmatrix}
\end{equation}
with $m_k$ the size of the Jordan block $J_k$.
The value of $f(A)$ is well-defined, if the derivatives of $f$ appearing in \eqref{eq:f_jordan} exist for all Jordan blocks of $A$. In this case, it is said that \emph{$f$ is defined on the spectrum of $A$}. This definition of $f(A)$ is equivalent to several other possible definitions of $f(A)$, e.g. through polynomial interpolation or Cauchy integrals \cite{Higham08}. \dgt{In the special case where $A$ is diagonalisable this reduces to 
\begin{align*}
	f(A)=Q^{-1}\diag(f(\xi_1),\dots,f(\xi_m))Q,
\end{align*}
consistent with the usual spectral calculus used for normal matrices $A$.}
If $f$ \dgt{is real analytic at a point $x_0$}, this definition also coincides with the result of `plugging $A$' into the Taylor expansion; for $f =\exp$, in particular, it coincides with the usual matrix exponential. 

Next, we review some basic results on Laplace transforms. Let $m: [0,\infty) \to \bC$ be a function of locally finite variation with $m(\infty) < \infty$ and let 
\begin{equation}\label{eq:Laplace}
\hat m(\lambda)  = \mathcal{L}[m](\lambda) = \int_0^\infty e^{-\lambda t}dm(t)
\end{equation}
denote its Laplace transform with complex argument $\lambda$. Following \cite[Ch.~II]{Widder46} there exists $\zeta_c: -\infty \le \zeta_c \le 0$, called \emph{abscissa of convergence}, such that \eqref{eq:Laplace} converges absolutely for all $\lambda \in \mathbb{H}_{\zeta_c}^+ = \set{z \in \bC: \Re z > \zeta_c}$ and diverges for all $\lambda \in \mathbb{H}_{\zeta_c}^- = \set{z \in \bC: \Re z < \zeta_c}$. Moreover, the function $\hat m$ is analytic in $\mathbb{H}_{\zeta_c}^+$. 

Finally, we show the following result on Laplace transforms evaluated at matrix arguments:
\begin{lem} \label{lem:matrix_laplace} Let $\hat m$ be the Laplace transform of $m$ with abscissa of convergence $\zeta_c \le 0$, and let $A \in \bR^{N \times N}$. If any of the following conditions holds:
\begin{enumerate}[(i)]
\item $\pi(A) + \zeta_c < 0$,
\item $\pi(A) < 0$,
\item $A$ is zero-stable,
\end{enumerate}
then $\hat m$ is defined on the spectrum of $A$ and 
\begin{equation}\label{eq:Laplace_at_A}
\int_0^\infty e^{At} dm(t) = \hat m(-A).
\end{equation}
\end{lem}
\begin{proof}
Assume (i). Then there exists $\epsilon > 0$ such that $\zeta_c < -(\pi(A) + \epsilon)$. To establish convergence, we use \Cref{lem:A_estimate} to estimate
\[\int_0^\infty \norm{e^{At}} dm(t) \le M_\epsilon \int_0^\infty e^{(\pi(A) + \epsilon)t} dm(t).\]
Since $-(\pi(A) + \epsilon)$ is larger than the abscissa of convergence $\zeta_c$ the right hand side is finite. To show the equality \eqref{eq:Laplace_at_A}, let $A = Q^{-1} J Q$ be the canonical Jordan form of $A$. Multiplying \eqref{eq:Laplace_at_A} with $Q^{-1}$ from the left and with $Q$ from the right, it is sufficient to show \eqref{eq:Laplace_at_A} for all Jordan blocks of $A$, i.e. to show
\[\int_0^\infty e^{J_k(\xi) t} dm(t) = \hat m(-J_k(\xi)),\]
for every eigenvalue $\xi \in \set{\xi_1,\dots, \xi_m}$ and associated Jordan block $J_k(\xi)$. Now $-\Re \xi \ge - \pi(A) > \zeta_c$, i.e. $\hat m$ is analytic in a neighborhood of $-\xi$. Hence, by \cite[Thm.~II.5a]{Widder46}, $\hat m^{(j)}(\xi)$ exists for any $j \in \bN_0$, and 
\[\hat m^{(j)}(-\xi) 
= \int_0^\infty (-t)^j e^{-\xi t} dm(t),\]
showing, element-by-element, the equality \eqref{eq:Laplace_at_A} for each Jordan block.
To show the claim under assumption (ii) it is sufficient to note that (ii) implies (i), because $\zeta_c \le 0$. To show the claim under 
assumption (iii), note that convergence of the integral can be established under the simpler estimate 
\[\int_0^\infty \norm{e^{At}} dm(t) \le M \left(m(\infty) - m(0) \right).\]
As for the equality \eqref{eq:Laplace_at_A}, the same argument as above applies to every eigenvalue $\xi$ with $\Re \xi < 0$. For the remaining eigenvalue $\xi = 0$, which must be simple, there is a single associated Jordan block of size $1$. For this block \eqref{eq:Laplace_at_A} becomes the trivial identity
\[\int_0^\infty  e^0 dm(t) = \hat m(0). \qedhere\]
\end{proof}

\subsection{Bernstein functions and inverse subordinators}\label{sec:bernstein functions}
 
A function $f\colon(0,\infty)\to[0,\infty)$ is said to be Bernstein if it has derivatives of all orders and 
\begin{align*}
	(-1)^{n-1}f^{(n)}(x)\geq0,\quad\forall x>0,\ \forall n \in \bN.
\end{align*}
In this case, we will write $f\in\cB\cF$. Moreover, \cite[Theorem 3.2]{Schilling10} states that $f\in\cB\cF$ if and only if 
\begin{align}\label{eq:bernstein function}
	f(x)=a+bx+\int_0^\infty(1-e^{-sx})\,\nu(ds),\quad x>0,
\end{align}
where $a,b\geq0$ and $\nu$ is a non-negative measure on $(0,\infty)$ with tail $s\mapsto\bar\nu(s)=a+\nu(s,\infty)$ satisfying the integrability criterion 
\begin{align}\label{con:levy measure integrability}
	\int_0^\infty(s\wedge1)\,\nu(ds)<\infty.
\end{align} 
Finally, we remark that any $f \in \cB\cF$ has an analytic extension onto the right complex half-plane $\bH_0^+=\{\lambda\in\bC\colon\Re\lambda>0\}$, such that it can be evaluated at any $\lambda \in \bC$ with $\Re \lambda > 0$.  
\par 
A \emph{subordinator} is a non-decreasing L\'evy process without killing. It is well known that a function $f$ is Bernstein with coefficients $b$, $\nu$, and $a = 0$ in \eqref{eq:bernstein function} if and only if it is the Laplace exponent of a subordinator with L\'evy triplet $(b, 0,\nu)$ \cite[Theorem 1.2]{Bertoin99}. 
In the course of this work $\sigma^f=(\sigma^f_s)_{s\geq0}$ denotes a subordinator with triplet $(b, 0,\nu)$, i.e. 
\begin{align*}
	\bE[e^{-\lambda\sigma^f_s}]=e^{-sf(\lambda)},\quad s,\lambda>0,
\end{align*}
where $f$ has representation (\ref{eq:bernstein function}). We define the inverse of $\sigma^f$ as the hitting time process
	\begin{align}\label{eq:inverse subordinator}
		L^f_t=\inf\{s>0\mid \sigma^f_s> t\},\quad t\geq0,
	\end{align} 
which, if there is no ambiguity with respect to $f$, will be denoted by $L_t$, $t\geq0$. 
Let $l_t(ds)=\bP(L_t\in ds)$ denote the distribution of $L_t$ and let $g_s(dt)=\bP(\sigma^f_s\in dt)$ denote the transition probability of $\sigma^f_s$, $s,t\geq0$. 


\begin{prop}\label{prop:double laplace transform l}
		Let $f\in\cB\cF$ with triplet $(b,0,\nu)$ where $b\geq0$, $\nu(0,\infty)=\infty$, and $s\mapsto\bar\nu(s)=\nu(s,\infty)$ is absolutely continuous on $(0,\infty)$. 
		\begin{enumerate}[(i)]
			\item The distribution $l_t$, $t\geq0,$ has a density such that $l_t(ds)=l_t(s)\,ds$ and $l_t(s)=bg_s(t)+(\bar\nu*g_s)(t)$ where (with slight abuse of notation) $l_t(s)$ and $g_s(t)$ denote the densities of $l_t(ds)$ and $g_s(dt)$, respectively. Furthermore, for any $s\geq0$ the Laplace transform of $t \mapsto l_t(ds)$ is given by 
			\begin{align*}
				\hat l_\lambda(s)
				=\cL[t \mapsto l_t(s)](\lambda)
				=\frac{f(\lambda)}{\lambda}e^{-sf(\lambda)}.
			\end{align*}
			\item Let $A\in\bR^{n\times n}$ with index of stability $\pi(A)$ and set $c = f^{-1}(\max(\pi(A),0) \ge 0$. Then, for any $\lambda \in \bC$ with $\Re \lambda > c$ 
		\begin{align*}
\cL[s \mapsto \hat l_\lambda(s)](-A) = \int_0^\infty e^{sA} \hat l_\lambda(s) ds 
			=\frac{f(\lambda)}{\lambda(f(\lambda)I-A)}. 
		\end{align*}
		\end{enumerate}
	\end{prop}
	\begin{rem}
	Note that under the condition $\nu(0,\infty)=\infty$ the Bernstein function $f$ is a strictly increasing bijection from $[0,\infty)$ to $[0,\infty)$. Hence its inverse $f^{-1}$ and the value $c = f^{-1}(\max(\pi(A),0)$ are well-defined and $c$ must satisfy $c \ge 0$.
	\end{rem}
	\begin{proof}
	The proof of (i) can be found in \cite[Proposition 3.2]{Toaldo15}. In order to show (ii), we first estimate 		\begin{align*}
			\int_0^\infty\int_0^\infty\lVert e^{sA}e^{-\lambda t}l_t(s)\rVert\,dt\,ds
			&\leq\int_0^\infty\lVert e^{sA}\rVert\int_0^\infty e^{-\Re\lambda t}l_t(s)\,dt\,ds\nonumber \\
			&=\frac{f(\Re\lambda)}{\Re\lambda}\int_0^\infty \lVert e^{sA}\rVert e^{-sf(\Re\lambda)}\,ds,
		\end{align*}
using part (i) in the second equality. By assumption $\Re \lambda > f^{-1}(\max(\pi(A),0))$, or equivalently, $f(\Re \lambda) > \max(\pi(A),0)$. Let $\epsilon > 0$ be small enough that $f(\Re \lambda) > \max(\pi(A) + \epsilon,0)$. 
Making use of \Cref{lem:A_estimate} we can find $M_\epsilon > 0$ such that
		\begin{align*}
			\int_0^\infty\int_0^\infty\lVert e^{sA}e^{-\lambda t}l_t(s)\rVert\,dt\,ds
			\leq\frac{M_\epsilon f(\Re\lambda)}{\Re\lambda} \int_0^\infty e^{-s(f(\Re\lambda)- \pi(A)- \epsilon)}\,ds < \infty.
		\end{align*}
Hence, we have 
		\begin{align*}
	\cL[s \mapsto \hat l_\lambda(s)](-A)
			=\int_0^\infty\int_0^\infty e^{sA}e^{-\lambda t}l_t(s)\,dt\,ds
			&=\frac{f(\lambda)}{\lambda}\int_0^\infty e^{-s(f(\lambda)I-A)}\,ds\\
			&=\frac{f(\lambda)}{\lambda(f(\lambda)I-A)}. \qedhere
		\end{align*}
 		\end{proof}
\begin{example}\label{ex:Mittag Leffler}
	If $f(\lambda)=\lambda^\alpha$, for $\lambda\in(0,\infty)$ and $\alpha\in(0,1)$, $\sigma^f$ is the \emph{$\alpha$-stable subordinator} and its inverse $L_t$ is called the \emph{inverse $\alpha$-stable subordinator} which has the density
\begin{align}\label{eq:density inverse subordinator}
	l_t(x)=\frac{t}{\alpha}x^{-1-\frac{1}{\alpha}}g_1(tx^{-\frac{1}{\alpha}}),\quad x,t\geq0,
\end{align}
where $t\mapsto g_s(t)$ is the density of $\sigma^f_s$, which has Laplace transform $\hat g_s(\lambda)=\exp(-s\lambda^\alpha)$, $\lambda\in\bR$, $s\in(0,\infty)$ \cite[Corollary 3.1]{Meerschaert04}. Bingham \cite[Proposition 1(a)]{Bingham71} showed that the inverse $\alpha$-stable subordinator has a Mittag Leffler distribution: 
\begin{align*}
	\bE[e^{-sL_t}]=E_{\alpha}(-st^\alpha),\quad s,t\geq0,
\end{align*}
where $E_{\alpha}$, for $\alpha>0$, is the one parameter Mittag-Leffler function,
\begin{align}\label{def:mittag leffler function}
	E_{\alpha}(z)=\sum_{k=0}^\infty\frac{z^k}{\Gamma(\alpha k+1)},\quad z\in\bC
\end{align} 
and $\Gamma(\lambda)\coloneqq \int_0^\infty t^{\lambda-1}e^{-t}\,dt$ is the Gamma function. Note that $E_{\alpha}$ is an entire function. Therefore it is defined on the spectrum of any square matrix $A \in \bR^{N \times N}$ and can be evaluated by plugging $A$ into \eqref{def:mittag leffler function}.
\end{example}

\subsection{General time-fractional derivatives}
For $f\in\cB\cF$ with triplet $(b, 0,\nu)$ we define the generalized $f$-Caputo derivative of a suitable function $q$ as
\begin{align}\label{def:caputo derivative}
	\bD^f_tq(t)
	\coloneqq\frac{d}{dt}\left(bq(t)+\int_0^t\bar\nu(t-s)(q(s)-q(0))\,ds\right),\quad t>0.
\end{align}
We note that the integral in \cref{def:caputo derivative} 
is well defined if $q$ is absolutely continuous on $(0,t]$ for each $t\geq0$, that is, $q\in AC((0,t])$ for all $t>0$, cf. \cite[page 97]{Kilbas06}. In this case, we write $q\in AC_{\loc}((0,\infty))$ and note that $q$ is differentiable a.e. by Rademacher's Theorem, that is, $\frac{d}{dt}q(t)$ exists for a.e. $t\geq0$. 
\begin{lem}\label{lem:laplace of fractional operator}
Let $q$ be of exponential order, that is,
$\lvert q(t)\rvert\leq Me^{\lambda_0 t}$, for some $\lambda_0,M>0$ and all $t\geq0$.
Then 
	\begin{align*}
		\cL\left[t \mapsto \bD^f_t q(t)\right](\lambda)
		=f(\lambda)\hat q(\lambda)-\frac{f(\lambda)}{\lambda}q(0)
	\end{align*}
exists with an abscissa of convergence $\zeta_c \le \lambda_0$. 
\end{lem}
\begin{proof}
	Clearly, $q\in AC_\loc((0,\infty))$, so $\bD^f_tq(t)$ is well defined. Note that for \( f\in\cB\cF \) expressed as in \eqref{eq:bernstein function}, the function $g(x) = \frac{f(x)}{x}$, $x > 0$, can be represented as $g(x) = b + \int_0^\infty e^{-sx} \bar{\nu}(s) \, ds$. Then, by taking Laplace transforms of \eqref{def:caputo derivative},  we get 
	\begin{align*}
		\cL\left[t \mapsto \bD^f_t q(t)\right](\lambda)
		&=b\lambda\hat q(\lambda)-bq(0)+\left(\hat q(\lambda)-\frac{q(0)}{\lambda}\right)\lambda\hat \nu(\lambda)\\
		&=\left(b+\hat \nu(\lambda)\right)\hat q(\lambda)\lambda-\left(b+\hat\nu(\lambda)\right)q(0)\\
		&=g(\lambda)\hat q(\lambda)\lambda-g(\lambda)q(0)\\
		&=f(\lambda)\hat q(\lambda)-\frac{f(\lambda)}{\lambda}q(0).
	\end{align*}
	The fact that $\zeta_c \le \lambda_0$ follows from \cite[Thm.~II.2.1]{Widder46} applied to $q$. 
\end{proof}
\begin{example}\label{example:classical caputo}
	If $f(\lambda)=\lambda^\alpha$, $\lambda\in(0,\infty)$, $\alpha\in(0,1)$, \cref{eq:bernstein function} can be identified as 
	\begin{align*}
		\lambda^\alpha=\int_0^\infty(1-e^{-s\lambda})\frac{\alpha s^{-\alpha-1}}{\Gamma(1-\alpha)}\,ds,
	\end{align*}
	which implies $a=b=0$ and 
	\begin{align*}
		\nu(ds)=\frac{\alpha s^{-\alpha-1}}{\Gamma(1-\alpha)}\,ds,
	\end{align*}
	and therefore 
	\begin{align*}
		\bar\nu(s)
		=\int_s^\infty\frac{\alpha \xi^{-\alpha-1}}{\Gamma(1-\alpha)}\,d\xi
		=\frac{s^{-\alpha}}{\Gamma(1-\alpha)}.
	\end{align*}
	Performing these substitutions in \cref{def:caputo derivative} shows that 
	\begin{align*}
	\bD^f_tq(t)=\bD^\alpha_tq(t),
	\end{align*}
	where 
	\begin{align*}
		\bD^\alpha_tq(t)
		=\frac{1}{\Gamma(1-\alpha)}\frac{d}{dt}\int_0^t(t-s)^{-\alpha}(q(s)-q(0))\,ds
	\end{align*}
	is the (conventional) Caputo derivative of order $\alpha$ \cite[Section 2.3]{Sikorskii12}.
\end{example}
\subsection{Polynomial processes}
In this section, we review the definition of a polynomial process, which is given in \cite{Cuchiero12} within the framework of time-homogeneous Markov processes. Throughout, let \( S \) be a closed subset of \( \mathbb{R}^d \), and let \( \mathcal{S} \) denote its Borel \( \sigma \)-algebra. \dbt{Let $(\Omega, \cF, \bP)$ be a measurable space, equipped with a right-continuous filtration $\bF = (\cF_t)_{t \geq 0}$. On this stochastic basis, we consider a time-homogeneous cadlag Markov process $X = (X_t)_{t \geq 0}$ with associated semigroup \( (P_t)_{t \geq 0} \), i.e. 
\begin{align}\label{eq:markov semi group}
    P_t f(x) \coloneqq \bE_x[f(X_t)], \quad x \in S,
\end{align}
acting on all Borel measurable functions \( f \colon S \to \mathbb{R} \) for which the integral is well defined. As usual, for any $x \in S$, we denote by $\bP_x(.) = \bP(.|X_0 = x)$ the conditional probability, given that $X$ starts in $x$. }
\dbt{We remark that in contrast to \cite{Cuchiero12} we do not allow for killing or explosion of $X$ and therefore do not attach a `cemetery state' to $S$.}\par 
 An operator $\cG$ 
	is called \emph{extended generator} for $X$ if its domain 
	$\cD_\cG$ consists of those Borel-measurable functions $u\colon S\to\bR$ for which there exists a function $\cG u$ such that $M^u$ defined by
	\begin{align*}
		M_t^u= u(X_t)-u(x)-\int_0^t\cG u(X_s)\,ds,\quad t\geq0,
	\end{align*}
	is an $\bP_x$-local martingale for every $x\in S$. Note that on bounded Borel functions, the extended generator coincides with the classical generator of the semigroup $(P_t)_{t \geq 0}$. 
\begin{defn}[Cf.~\cite{Cuchiero12}]\label{def:polynomial process}
	An $S$-valued time-homogeneous Markov process $X$ with extended generator $\cG$ is called \emph{$m$-polynomial} if for all $k\in\{1,\dots,m\}$, $u\in\cP_k$, and $x\in S$, it holds that 
	\begin{align*}
		P_t\lvert u\rvert(x)=\bE_x\left[\lvert u(X_t)\rvert\right]<\infty, \quad \text{and}\quad 
		\cG(\cP_k)\subseteq\cP_k.
	\end{align*}
	If $X$ is $m$-polynomial for every $m\geq0$, then it is called \emph{polynomial}.
\end{defn}
Examples of polynomial processes include Brownian motion, L\'evy processes (under suitable conditions on finiteness of moments), Gaussian and Non-Gaussian Ornstein-Uhlenbeck processes, the Cox-Ingersoll-Ross process, all Pearson diffusions, and many stochastic volatility models used in financial mathematics, such as the Heston model, the Bates model, and the Barndorff-Nielsen-Shepard model. 
	In the sequel, $X$ denotes an $m$-polynomial process with extended generator $\cG$ and $k\in\{1,\dots,m\}$. Using Markovian techniques and the fact that \( \mathcal{G} \) leaves \( \mathcal{P}_k \) invariant, \cite[Lemma 2.6]{Cuchiero12} show that \( q(t,x) := \mathbb{E}_x[u(X_t)] \) is the unique solution to the Kolmogorov backward equation
\begin{align}\label{eq:KBE}
		\begin{cases}
			\partial_tq(t,x)=\cG q(t,x),\quad0<t<\infty,\\
			q(0,\cdot)=u\in\cP_k. 
		\end{cases}
	\end{align}
Given a basis\footnote{Typically, a monomial basis is used here, e.g. $(1,x,y,x^2,xy,y^2)$ for $\cP_2$ over $S = \bR^2$.} $H=(h_1,\dots,h_N)$ of $\cP_k$ where $N=\dim\cP_k$, the restriction of $\cG$ to $\cP_k$ can be represented in the basis $H$ by a matrix $A \in \bR^{N \times N}$, i.e., we have
	\begin{align*}
		\cG h_i\eqqcolon\sum_{j=1}^NA_{ij}h_j \qquad \text{for all } i=1, \dotsc, N.
	\end{align*}
In the sequel, we denote this matrix as $A=\cG|_{\cP_k}$. Similarly, we write $\vec{u} \in \bR^N$ for the coordinate representation of $u \in \cP_k$ in the basis $H$. On $\cP_k$, (\ref{eq:KBE}) is then equivalent to the vector-valued linear ordinary differential equation 
	\begin{align*}
		\begin{cases}
			\partial_tq(t)=Aq(t),\quad0<t<\infty,\\
			q(0)=\vec{u}\in\bR^N,
		\end{cases}
	\end{align*}
	whose unique solution is given by $t\mapsto e^{tA}\vec{u}$ \cite[Theorem 2.9]{Engel01}. This yields the following property of polynomial processes. 
\begin{thm}[Moment formula, cf.~\cite{Cuchiero12}]\label{thm:moment formula for polynomial processes}
	Let $X$ be $m$-polynomial with extended generator $\cG$. 
	For each $k\in\{1,\dots,m\}$ we find $A\in \bR^{N \times N}, N = \dim\cP_k$, such that $P_t|_{\cP_k}=e^{tA}$ for all $t\geq0$, i.e. for all $x\in S$ and $u\in\cP_k$, 
	\begin{align*}
		P_tu(x)=\bE_x[u(X_t)]=H(x)^\top e^{tA}\vec{u},\quad t\geq0. 
	\end{align*}
\end{thm}
\begin{rem}	
\begin{enumerate}[(i)]
    \item Taking advantage of the fact that polynomial processes are special semimartingales with polynomial coefficients, \Cref{thm:moment formula for polynomial processes} can be used to characterise polynomial processes. A detailed exposition of this can be found in \cite{Cuchiero12}. 
    \item An important extension of the moment formula is presented in \cite[Theorem 2.5]{Filipovic17}, where the finiteness of absolute moments in \Cref{def:polynomial process} is not required. In this case, the conditional moment formula holds, i.e. for $u\in\cP_k$,
    \begin{align}\label{eq:conditional moment formula}
        \mathbb{E}[u(X_T) \mid \cF_t] = H(X_t)^\top e^{(T-t)A} \vec{u}, \quad t \leq T.
    \end{align}
\end{enumerate}
\end{rem}

\subsection{Stochastic Representations} 
\dbt{
We conclude the section with some stochastic representations of polynomial processes, which can be found in \cite{Cuchiero12} and \cite{Filipovic16}. The first represents the process in terms of semimartingale characteristics, while the second provides an SDE in the diffusion case. Note that further representations can be found in the literature; see, e.g. \cite{Filipovic17} for the `jump-diffusion' case. We will compare these representations to the subordinated case in \Cref{sub:stochastic_sub}.}
\begin{prop}\label{prop:stochastic}
\dbt{
\begin{enumerate}[(a)]
\item Let $X$ be a polynomial process on $S$ without killing or explosion. Then, on any of the filtered probability spaces $(\Omega, \cF, \bF, \bP_x), x \in S$, the process $X$ is a special Ito semimartingale, and its characteristics $(B, C, \eta)$ associated with the ``truncation function'' $\chi(\xi) =\xi$ satisfy 
\begin{align}
 &B_{t,i}=\int_0^{t} \mathfrak{b}_i(X_{s}) ds,\label{eq:B}\\
 & C_{t,ij}+\int_0^t\int_{\mathbb{R}^n}\xi_i\xi_j \eta(ds, d\xi)= \int_0^{t}  \mathfrak{a}_{ij}(X_{s}) ds,\label{eq:C}
\end{align}
where $\mathfrak{b}_i \in \cP_1$ and
$\mathfrak{a}_{ij} \in \cP_2$, for all $i,j \in \set{1, \dotsc, d}$. Moreover, $C$ and $\eta$ can be written as
\begin{align}\label{eq:C_eta}
C_{t,ij}=\int_{0}^{t}c_{s,ij}ds, \quad \eta(\omega; [0,t], d\xi)=\int_0^t K_{\omega,s}(d\xi)ds,
\end{align}
where $(c_{ij})_{i,j \leq n}$ is a predictable process and $K_{\omega,t}(d\xi)$ is a predictable random measure on $(\mathbb{R}^n, \mathcal{B}(\mathbb{R}^n))$, which satisfies 
\begin{align}\label{eq:K}
\int_{\mathbb{R}^{n}} \xi^{\mathbf{k}} K_{\omega,t}(d\xi) =\sum_{|\mathbf{l}|=0}^{|\mathbf{k}|}\alpha_{\mathbf{l}} X_t^{\mathbf{l}}(\omega), \quad \textrm{for almost all } t \geq 0,
\end{align}
with some coefficients $\alpha_{\mathbf{l}}$, for every multi-index $\mathbf{k} \in \bN_0^d$ with $|\mathbf{k}| \ge 3$.
\item Let $X$ be a polynomial diffusion (i.e., a diffusion process which is also a polynomial process), then it satisfies the SDE
\begin{equation}\label{eq:SDE}
dX_t = \mathfrak{b}(X_t)\,dt + \sigma(X_t)\, dW_t,
\end{equation}
\end{enumerate}
where $W$ is a $d$-dimensional Brownian motion, and where, for every $i,j \in \set{1, \dotsc, d}$, it holds that $\mathfrak{b}_i \in \cP_1$ and
$\mathfrak{a}_{ij}  = (\sigma \sigma^\top)_{ij} \in \cP_2$.}
\end{prop}
\begin{proof}
\dbt{
Assertion (a) follows from \cite[Prop.~2.12]{Cuchiero12} and assertion (b) from \cite[Lem~2.2]{Filipovic16}.}
\end{proof}

\section{Generalized fractional polynomial processes}\label{sec:inverse subordinated polynomial processes}
\dbt{
Recall, that we work on the probability space $(\Omega, \cF, \bP)$ equipped with a right-continuous filtration $\bF = (\cF_t)_{t \geq 0}$, and carrying a polynomial process $X = (X_t)_{t \geq 0}$. Moreover, let $\sigma^f = (\sigma_t^f)_{t \geq 0}$ be a L\'evy subordinator with Laplace exponent $f \in \cB\cF$, independent from $X$, and adapted to $\bF$. Its inverse process $L = (L_t)_{t \geq 0}$ is then an increasing process with the property that each $L_t$ is a $\bF$-stopping time. Hence, the time-changed filtration $\bH = (\cH_t)_{t \geq 0}$, is well defined by setting $\cH_t = \cF_{L_t}$, and the process $Y_t = X_{L_t}$ -- our main object of interest -- is $\bH$-adapted.}\par
\dbt{We will derive a generalization of Kolmogorov's backward equation for the time-changed semi-group $\cT_tu(.) =\bE_.[u(X_{L_t})]$ associated to $Y$ in \Cref{sub:bw_eq} and show that $\cT_t$ preserves polynomials. Then, in \Cref{sub:moment}, we provide a moment formula for $Y$ in analogy to \Cref{thm:moment formula for polynomial processes}. Finally, we give two stochastic representations of $Y$, parallel to \Cref{prop:stochastic}, in \Cref{sub:stochastic_sub}. }
\subsection{Generalized fractional Kolmogorov backward equation}\label{sub:bw_eq}
Let $\cG$ denote the extended generator of the polynomial process $X$. We present the generalized fractional analogue to \cref{eq:KBE}, which, as we will show, 
serves as the Kolmogorov backward equation for inverse subordinated polynomial processes.  
Here, the ordinary time derivative is replaced by the generalized $f$-Caputo derivative defined in \cref{def:caputo derivative} with respect to $f \in \cB\cF$ with triplet $(b,a,\nu)$, i.e.
\begin{align*}
	\begin{cases}
		\bD^f_tq(t,x)=\cG q(t,x),\quad0<t<\infty,\\
		 q(0,\cdot)=u\in\cP_k.
	\end{cases}
\end{align*}
We set $A=\cG|_{\cP_k}\in\bR^{N\times N}$ where $N=\dim\cP_k$. Then on $\cP_k$, above problem is equivalent to the vector-valued linear generalized fractional differential equation
\begin{align}\label{problem:linear FDE1}
	\begin{cases}
		\bD^f_tq(t)=Aq(t),\quad 0<t<\infty,\\
		q(0)=\vec{u}\in\bR^N.
	\end{cases}
\end{align}
	\begin{thm}\label{thm:timechangedSG}
		Let $A \in \bR^{N \times N}$ and let $\sigma^f$ be a subordinator with Laplace exponent $f$ with representation (\ref{eq:bernstein function})
		where $a=0$, $b\geq0$, $\nu(0,\infty)=\infty$, and $s\mapsto\bar\nu(s)=\nu(s,\infty)$ is absolutely continuous on $(0,\infty)$. Further, let $L_t$, $t\geq0$ be the inverse of $\sigma^f$ and 
		define for each $\vec{u}\in\bR^N$ and $t\geq0$ the linear mapping
		\begin{align*}
			T_t\vec{u}=\bE[e^{L_tA}]\vec{u}.
		\end{align*}
		Then 
		the following holds: 
		\begin{enumerate}[(i)]
			\item $T_t$ is well-defined for all $t \geq 0$;
			\item $(t\mapsto T_t\vec{u})\in AC_\loc((0,\infty))$;
			\item $t\mapsto T_t\vec{u}$ uniquely solves (\ref{problem:linear FDE1}).
		\end{enumerate}
	\end{thm}
	\begin{proof} For (i) we have to show that 
	\begin{align*}
		T_t\,
		=\bE[e^{L_tA}]
		=\int_0^\infty e^{sA}l_t(s)\,ds
	\end{align*}
	is finite for all $t \ge 0$. Here, $l_t(s)$ resp. $g_s(t)$ represents the density of $l_t(ds)$ resp. $g_s(dt)$ which exists because $\nu(0,\infty)=\infty$ and $s\mapsto\bar\nu(s)$ is absolutely continuous; see \Cref{prop:double laplace transform l} (i). 
		\par 
Using \Cref{prop:double laplace transform l}, we know that the integral
		\begin{align*}
			\cL[t \mapsto T_t](\lambda) =\int_0^\infty\int_0^\infty e^{-\lambda t}e^{sA}l_t(s)\,ds\,dt
		\end{align*}
		is absolutely convergent for $\lambda \in \bC$ with $\Re\lambda>c$, where $c = f^{-1}(\max(\pi(A),0)) \ge 0$. An application of Fubini's Theorem \cite[Corollary 14.9]{Schilling17} implies that the function $s\mapsto e^{-\lambda t}e^{sA}l_t(s)$ 
		is integrable in $\bR^{N\times N}$. Dropping the scaling factor $e^{-\lambda t}$ yields the integrability of $s\mapsto e^{sA}l_t(s)$ in $\bR^{N\times N}$, thus, $T_t = \bE[e^{L_tA}]$ is well-defined for all $t \ge 0$.
		\par 
	To (ii): first assume $b=0$. Then $l_t(s)=(\bar\nu*g_s)(t)$ by \Cref{prop:double laplace transform l} (i). For each $t\geq0$ and $\vec{u}\in\bR^n$ this gives 
		\begin{align*}
			T_t\vec{u}
			&=\int_0^\infty e^{sA}\vec{u}\int_0^t\bar\nu(t-r)g_s(r)\,dr\,ds\\
			&=\left(\int_0^t\bar\nu(t-r)\int_0^\infty e^{sA}g_s(r)\,ds\,dr\right)u\\
			&=(\bar\nu*F)(t)u,
		\end{align*}
		where we used Fubini' Theorem in the second equality and set 
		\begin{align*}
			F(r)=\int_0^\infty e^{sA}g_s(r)\,ds,\quad r\geq0.
		\end{align*}
		In order to see that $(t\mapsto T_t\vec{u})\in AC_\loc((0,\infty))$, it is sufficient to show that $F\in L^1_\loc((0,\infty))$ and $\bar\nu\in AC_\loc((0,\infty))$ \cite[Chapter 3.8, Corollary 7.4]{Gripenberg90}. Now $\bar\nu\in AC_\loc((0,\infty))$ by assumption. Hence it is left to show that $F\in L^1_\loc((0,\infty))$: for each $t\geq0$, we have to show that $I_t := \lVert F\rVert_{L^1((0,t))}$ is finite. We first estimate
		\begin{align*}
			I_t = \int_0^t\lVert F(r)\rVert\,dr &\leq\int_0^t\int_0^\infty\lVert e^{sA}g_s(r)\rVert\,ds\,dr\\
			&=\int_0^\infty\lVert e^{sA}\rVert\int_0^tg_s(r)\,dr\,ds\\
			&=\int_0^\infty\lVert e^{sA}\rVert\int_s^\infty l_t(r)\,dr\,ds\\
			&=\int_0^\infty l_t(r)\int_0^r\lVert e^{sA}\rVert\,ds\,dr,
		\end{align*}
		where we have used that $\bP(\sigma_s\leq t)=\bP(L_t\geq s)$ for all $s,t\geq0$. Now choose $\epsilon > 0$ and set $c_\epsilon = \max(\pi(A) + \epsilon,0)$. For any $\lambda > f^{-1}(c_\epsilon)$, we have
		 
		\begin{align}
			\int_0^\infty e^{-\lambda t} I_t dt &\leq M_\epsilon \int_0^\infty \int_0^\infty e^{-\lambda t} l_t(r) \int_0^r e^{s c_\epsilon} ds\,dr\,dt \\
			&\leq \frac{M_\epsilon}{c_\epsilon} \int_0^\infty \int_0^\infty e^{-\lambda t} l_t(r)\,dt e^{r c_\epsilon}\,dr \notag\\
			&= \frac{M_\epsilon}{c_\epsilon} \frac{f(\lambda)}{\lambda} \int_0^\infty e^{r (c_\epsilon - f(\lambda)}\,dr < \infty, \label{eq:F is L1}
		\end{align}
		which proves that $F\in L^1_\loc((0,\infty))$.
		 \par 
		Let us now consider the case $b>0$: 
		since $t\mapsto e^{tA}$ is continuous and locally bounded it is locally Lipschitz continuous on $(0,\infty)$, i.e. for $u\in\bR^N$ and every $z>t_2>t_1\geq0$, we have that
		\begin{align*}
			\lVert e^{t_2A}u-e^{t_1A}u\rVert
			\leq\int_{t_1}^{t_2}\lVert \partial_s e^{sA}u\rVert\,ds
			&=\int_{t_1}^{t_2}\lVert Ae^{sA}u\rVert\,ds\\
			&=\int_{t_1}^{t_2}\lVert e^{sA}(Au)\rVert\,ds\\
			&\leq c(z)\lVert Au\rVert \lvert t_2-t_1\rvert.
		\end{align*}
	    Note that when $b>0$, $\lvert L_t-L_s\rvert\leq\lvert t-s\rvert/b$. Hence for every $z>t>s\geq0$ we get
		\begin{align*}
			\lVert T_t\vec{u}- T_s\vec{u}\rVert
			=\lVert\bE[e^{L_tA}\vec{u}-e^{L_sA}\vec{u}]\rVert
			\leq c_1(z)\bE[\lvert L_t-L_s\rvert]
			\leq c_1(z)(t-s)/b.
		\end{align*}
		Therefore, $t\mapsto T_t\vec{u}$ is locally Lipschitz continuous which directly implies that $t\mapsto T_t\vec{u}$ is locally absolutely continuous on $(0,\infty)$. \par 		
		To (iii): the map $t\mapsto\bD^f_tq(t)$ is well defined for $q(t)=T_t\vec{u}$, $t\geq0$, by (ii) and the remarks after \cref{def:caputo derivative}. We proceed with showing that $t\mapsto T_t\vec{u}$ uniquely solves \cref{problem:linear FDE1}.	
		By \Cref{lem:laplace of fractional operator} we note that the Laplace transform of \cref{problem:linear FDE1} becomes 
			\begin{align}\label{eq:laplace transformed problem}
				\begin{cases}
					f(\lambda)\hat q(\lambda)-\frac{f(\lambda)}{\lambda}q(0)=A \hat q(\lambda),\quad 0<t<\infty,\\
					q(0)=\vec{u}\in\bR^N.
				\end{cases}
			\end{align} 
			In view of \Cref{prop:double laplace transform l} the Laplace transform of $t\mapsto T_t\vec{u}$ reads
			\begin{align*}
				\hat{T}_\lambda\vec{u}
			=\int_0^\infty e^{-\lambda t}T_t\vec{u}\,dt
			&=\left(\int_0^\infty e^{sA}\int_0^\infty e^{-\lambda t}l_t(s)\,dt\,ds\right)\vec{u}\\
			&=\doublehat{l}_\lambda(-A)\vec{u},
			\end{align*}
			for $\Re\lambda>c$ where $c = f^{-1}(\max(\pi(A),0)) \ge 0$.
			In this case, $\hat T_\lambda\vec{u}$ is absolutely convergent as remarked in \Cref{prop:double laplace transform l} (ii) and we derive 
			\begin{align*}
				A\hat{T}_\lambda\vec{u}
				&=\int_0^\infty e^{-\lambda t}AT_t\vec{u}\,dt
				=\int_0^\infty e^{-\lambda t}\int_0^\infty Ae^{sA}\vec{u}\,l_t(s)\,ds\,dt\\
				&=\int_0^\infty e^{-\lambda t}\int_0^\infty \frac{d}{ds}e^{sA}\vec{u}\,l_t(s)\,ds\,dt\\
				&=\frac{f(\lambda)}{\lambda}\int_0^\infty e^{-sf(\lambda)}\frac{d}{ds}e^{sA}\vec{u}\,ds\\
				&=\frac{f(\lambda)}{\lambda}\lim_{h\to0}\left(\frac 1h\int_0^\infty e^{-sf(\lambda)}e^{(s+h)A}\vec{u}\,ds-\frac 1h\int_0^\infty e^{-s(f(\lambda)-A)}\vec{u}\,ds\right)\\
				&=\frac{f(\lambda)}{\lambda}\lim_{h\to0}\left(\frac{e^{hf(\lambda)}}{h}\int_h^\infty e^{-sf(\lambda)}e^{sA}\vec{u}\,ds
				-\frac 1h\int_0^\infty e^{-s(f(\lambda)-A)}\vec{u}\,ds\right)\\
				&=\frac{f(\lambda)}{\lambda}\lim_{h\to0}\left(\frac{e^{hf(\lambda)}-1}{h}\int_0^\infty e^{-sf(\lambda)}e^{sA}\vec{u}\,ds
				-\frac{e^{hf(\lambda)}}{h}\int_0^h e^{-s(f(\lambda)-A)}\vec{u}\,ds\right)\\
				&=\frac{f(\lambda)}{\lambda}\left(f(\lambda)\frac{\lambda}{f(\lambda)}\hat{T}_\lambda\vec{u}
				-e^{0f(\lambda)}e^{0(f(\lambda)-A)}\vec{u}\right)\\
				&=f(\lambda)\hat{T}_\lambda\vec{u}-\frac{f(\lambda)}{\lambda}\vec{u},
			\end{align*}
			where in the fourth step we used \Cref{prop:double laplace transform l} (i). Thus, $\hat{T}_\lambda\vec{u}$ solves \cref{eq:laplace transformed problem} for all $\lambda\in\bC$ for which $\Re\lambda>c$, and from the uniqueness of the Laplace transform \cite[Theorem 1a, p. 432]{Feller71} it follows that  $t\mapsto T_t\vec{u}$ is the unique continuous 
			solution to \cref{problem:linear FDE1}.  		
	\end{proof}
\begin{cor}\label{cor:solution pol}
	Suppose the assumptions in \Cref{thm:timechangedSG} are satisfied and let $X$ be an $m$-polynomial process. Let $A \in \bR^{N \times N}$, $N = \dim \cP_k$, be its generator matrix on $\cP_k$, $k \le m$ and, for each $u\in\cP_k$ and $t\geq0$, define the operator 
	\begin{align}\label{eq:time-changed semigroup}
		\cT_tu(x)
		=\bE_x[u(X_{L_t})],\quad x\in S.
	\end{align}
	Then 
		the following holds: 
		\begin{enumerate}
			\item[(i)] $\cT_t$ preserves polynomials, i.e. $\cT_t(\cP_k)\subseteq\cP_k$ for all $t\geq0$; 
			\item[(ii)] $t\mapsto\overrightarrow{\cT_tu}$ uniquely solves (\ref{problem:linear FDE1}).
		\end{enumerate}
\end{cor}
\begin{proof}
	Since $X$ is $m$-polynomial and $\cP_k\cong\bR^N$ we can write
	\begin{align*}
		\cT_tu(x)
		=\bE_x\left[H(X_{L_t})^\top\vec{u}\right]
		&=\bE\left[\bE_x\left[H(X_{L_t})^\top\vec{u}\mid\cF_{L_t}\right]\right]\\
		&=H(x)^\top\bE\left[e^{L_tA}\vec{u}\right],
	\end{align*}
	where $\vec{u}\in\bR^N$. Now \Cref{thm:timechangedSG} (i) yields that for $u\in\cP_k$ and each $t\geq0$,
	\begin{align*}
		\overrightarrow{\cT_tu}
		=T_t\vec{u}
		=\bE\left[e^{L_tA}\vec{u}\right]\in\bR^N,
	\end{align*} 
	and hence, $\cT_tu\in\cP_k$. Then \Cref{thm:timechangedSG} (ii) and (iii) prove the remaining assertion.
\end{proof}
\begin{rem}
	\dgt{
	\begin{enumerate}
		\item Note that $(\mathcal{T}_t)_{t\geq0}$, as defined in (\ref{eq:time-changed semigroup}), generally does not inherit the semigroup property of  $(P_t)_{t\geq0}$ (see (\ref{eq:markov semi group})), unless in special cases, such as when the time-change $(L_t)_{t\geq0}$ is deterministic.
		\item The findings in \Cref{cor:solution pol} can be extended to an infinite-dimensional setting of polynomial processes in Banach spaces, as introduced in \cite{Cuchiero21,Cuchiero24}. In this framework, conditional moments for polynomial processes are expressed through a deterministic dual process which is the solution of a linear ODE. In a similar approach, replacing the linear ODE by a linear FDE yields the conditional moments of the corresponding time-changed polynomial process. The approach is generic and primarily requires weak formulations of \Cref{problem:linear FDE1} and suitable solution concepts in Banach spaces. 
	\end{enumerate}
	}	
\end{rem}
\subsection{A moment formula}\label{sub:moment}
We are now in a position to generalize the moment formula in \Cref{thm:moment formula for polynomial processes} for conventional polynomial processes to polynomial processes time-changed by an inverse subordinator $L_t$. 
\begin{thm}[Moment Formula I]\label{cor:gen_moment_formula}Let $f \in \cB\cF$ with triplet $(b,0,\nu)$ and let $L_t$ be the corresponding inverse subordinator. Let $\bD_t^f$ be the associated $f$-Caputo derivative, and let $q(t,a)$ be the solution of the \emph{scalar linear fractional differential equation}
\begin{equation}\label{problem:linear_FDE_scalar}
		\bD^f_tq(t,a)=a q(t,a),\qquad q(0)=1.
\end{equation}
Moreover, let $X$ be $m$-polynomial with extended generator $\cG$ and let $A \in \bR^{N \times N}$ be the matrix representation of $\cG$ in a basis $H(x)$ of $\cP_m$. Then, for any $t \ge 0$ the mapping $a \mapsto q(t,a)$ is defined on the spectrum of $A$ and for any $u \in \cP_m$, we have
	\begin{align*}
		\cT_tu(x) =\bE_x[u(X_{L_t})]=H(x)^\top q(t,A) \vec{u}. 
	\end{align*}
\end{thm}
In other words, whenever we can solve the linear fractional scalar differential equation \eqref{problem:linear_FDE_scalar}, we can also give a formula for \emph{moments of arbitrary order} of $X_{L_t}$. 
\begin{proof}
Let $q(t,-a) = \bE[e^{-aL_t}]$ be the Laplace transform of $L_t$. By \Cref{thm:timechangedSG}(i) $q(t,a)$ is well defined for all $t \ge 0$ and $a \in \bR$. Hence, the abscissa of convergence of $\bE[e^{-aL_t}]$ is $\zeta_c = -\infty$, i.e., $q(t,a)$ is an entire function for every $t \ge 0$. By \Cref{lem:matrix_laplace} it follows that $q(t,A)$ is well defined for every matrix $A \in \bR^{N \times N}$ and $q(t,A) = \bE[e^{L_tA}]$. The moment formula now follows from \Cref{thm:timechangedSG} in combination with \Cref{cor:solution pol}.
\end{proof}

\subsection{Stochastic Representations} \label{sub:stochastic_sub}
\dbt{
We provide some stochastic representations of the subordinated polynomial process $Y_t = X_{L_t}$. Recall that $X$ is a semimartingale with respect to $\bF = (\cF_t)_{t \geq 0}$ and that $\bH$ denotes the time-changed filtration defined by $\cH_t = \cF_{L_t}, t \geq 0$.}
\begin{prop}\dbt{Let $f \in \cB\cF$ with triplet $(b, 0,\nu)$, satisfying $\nu(0,\infty) = \infty$ or $b > 0$, and let $L_t$ be the corresponding inverse subordinator.
\begin{enumerate}[(a)]
\item On any of the filtered probability spaces $(\Omega, \cF, \bH, \bP_x), x \in S$, the process $Y_t = X_{L_t}$ is a special semimartingale, and its characteristics $(B', C', \eta')$ associated to the ``truncation function'' $\chi(\xi) =\xi$ satisfy 
\begin{align}\label{eq:B_sub}
 &B'_{t,i}=\int_0^{t} \mathfrak{b}_i(Y_{s}) dL_s,\\
 & C'_{t,ij}+\int_0^t\int_{\mathbb{R}^n}\xi_i\xi_j \nu'(ds, d\xi)= \int_0^{t}  \mathfrak{a}_{ij}(Y_{s}) dL_s,\label{eq:C_sub}
\end{align}
where $\mathfrak{b}_i \in \cP_1$ and
$\mathfrak{a}_{ij} \in \cP_2$, for all $i,j \in \set{1, \dotsc, d}$. Moreover, $C$ and $\nu$ can be written as
\begin{align}\label{eq:C_eta_sub}
C'_{t,ij}=\int_{0}^{t}c_{s,ij}dL_s, \quad \eta'(\omega; [0,t], d\xi)=\int_0^t K'_{\omega,s}(d\xi)dL_s,
\end{align}
where $(c_{ij})_{i,j \leq n}$ is a predictable process and $K'_{\omega,t}(d\xi)$ is a predictable random measure on $(\mathbb{R}^n, \mathcal{B}(\mathbb{R}^n))$, which satisfies 
\begin{align}\label{eq:K sub}
\int_{\mathbb{R}^{n}} \xi^{\mathbf{k}} K'_{\omega,t}(d\xi) =\sum_{|\mathbf{l}|=0}^{|\mathbf{k}|}\alpha_{\mathbf{l}} Y_t^{\mathbf{l}}(\omega), \quad \textrm{for almost all } t \geq 0,
\end{align}
with some coefficients $\alpha_{\mathbf{l}}$, for every multi-index $\mathbf{k} \in \bN_0^d$ with $|\mathbf{k}| \ge 3$.
\item If $b = 0$, then $Y$ is not an Ito semimartingale; more precisely the probability that the characteristics $(B', C', \nu')$ are absolutely continuous is zero. 
\item Let $X$ be a polynomial diffusion as in \Cref{prop:stochastic}(b). Then $Y_t = X_{L_t}$ satisfies the SDE
\begin{equation}\label{eq:SDE_sub}
dY_t = \mathfrak{b}(Y_t)dL_t + \sigma(Y_t) dW_{L_t},
\end{equation}
\end{enumerate}
where $W$ is a $d$-dimensional Brownian motion, and where, for every $i,j \in \set{1, \dotsc, d}$, it holds that $\mathfrak{b}_i \in \cP_1$ and
$\mathfrak{a}_{ij}  = (\sigma \sigma^\top)_{ij} \in \cP_2$.
}
\end{prop}
\begin{proof}
Following \cite{Kobayashi2011}, a process $X$ is \emph{in synchronization} with a time-change $L$, if $X$ is a.s.\ constant on every stochastic interval $[L_{t-},L_t]$. The condition that $\nu(0,\infty) = \infty$ or $b > 0$ guarantees that the associated subordinator $\sigma^f$ is a strictly increasing L\'evy process, cf. \cite{Bertoin99}. Thus, its inverse $L$ has continuous paths, and it follows that any process $X$ is in synchronization with $L$. In the language of \cite{Jacod03}, \emph{``$X$ est adapt\'e $L$''} and we can apply \cite[Thm.~10.17]{Jacod03} and \cite[Thm.~10.19]{Jacod03} to obtain \eqref{eq:B_sub} and the first part of \eqref{eq:C_eta_sub} from \Cref{prop:stochastic}. Regarding the third characteristic $\eta'$, \cite[Thm.~10.27]{Jacod03} and \eqref{eq:C_eta} yield
\begin{align*}
\eta'(\omega, [0,t], d\xi) = \eta(\omega, [0,L_t], d\xi) = \int_0^{L_t} K_{\omega,s}(d\xi) ds = \int_0^t K_{\omega,L_s} dL_s.
\end{align*}
Setting $K'_{\omega,s}(d\xi) = K_{\omega,L_s}(d\xi)$ we obtain the second part of \eqref{eq:C_eta_sub}. Now, using \eqref{eq:K}, it holds for any multi-index $\mathbf{k} \in \bN_0^d$ that
\begin{align*}
\int_{\mathbb{R}^{n}} \xi^{\mathbf{k}} K'_{\omega,t}(d\xi) &= \int_{\mathbb{R}^{n}} \xi^{\mathbf{k}} K_{\omega,L_t}(d\xi) = \sum_{|\mathbf{l}|=0}^{|\mathbf{k}|}\alpha_{\mathbf{l}} X_{L_t}^{\mathbf{l}}(\omega) = 
\sum_{|\mathbf{l}|=0}^{|\mathbf{k}|}\alpha_{\mathbf{l}} Y_{t}^{\mathbf{l}}(\omega),
\end{align*}
which is \eqref{eq:K sub}. A similar calculation, reaplacing $\xi^{\mathbf{k}}$ by $\xi_i\xi_j$ shows \eqref{eq:C_sub}, completing part (a).
\par 
\dbt{To show (b), suppose that $b > 0$. Then, by \cite[Prop~1.8]{Bertoin99}, the range of the subordinator $\sigma^f$ has Lebesgue measure zero a.s. The range of $\sigma^f$ is equal to the support of the Stieltjes measure $dL_t$, cf. \cite[Prop~1.8]{Bertoin99}, and hence $L$ cannot be absolutely continuous. Thus $Y$ is not an Ito semimartingale.}\par
\dbt{Assertion (c) follows from \cite[Thm.~4.2]{Kobayashi2011}, using the fact that $X$ is in synchronization with $L$.}
\end{proof}

\section{Fractional polynomial processes}
In this section, we assume that $X$ is $m$-polynomial and that $L_t$, $t \geq 0$, is the inverse $\alpha$-stable subordinator for $\alpha \in (0,1)$, as discussed in \Cref{ex:Mittag Leffler}. For this setting, we present a moment formula for $(X_{L_t})_{t \geq 0}$ in \Cref{sec:moment formula} which serves as pendant to \Cref{thm:moment formula for polynomial processes} for inverse $\alpha$-stable subordinated polynomial processes. Additionally, we discuss cross-moments in equilibrium and their long-range dependence, as well as the underlying correlation structure of the time-changed process in \Cref{sec:cross moments} and \Cref{sec:correlation}, respectively.
\subsection{Moment formula}\label{sec:moment formula}
According to \Cref{cor:gen_moment_formula}, moments of inverse subordinated polynomial processes can be determined as solutions to linear fractional differential equations. For \(\alpha\)-stable subordination, the following corollary provides an analytic expression for computing these moments.

\begin{thm}[Moment formula II]\label{cor:moment formula for time-changed pol processes}
\dbt{Let $L$ be the inverse process of an $\alpha$-stable subordinator $\sigma^f$ with $\alpha \in (0,1)$.} Let $X$ be $m$-polynomial with generator $\cG$ and let $A \in \bR^{N \times N}$ be the matrix representation of $\cG$ in a basis $H(x)$ of $\cP_m$. For all $p\in\cP_m$ and $x\in S$ we have
\begin{align*}
	\bE_x\left[p(X_{L_t})\right]
	=H(x)^\top E_{\alpha}(t^\alpha A)\vec{p},\quad t>0,
\end{align*}
where $E_{\alpha}(\cdot)$, for $\alpha>0$, is the Mittag Leffler function defined in \cref{def:mittag leffler function}.
\end{thm}
\begin{proof}
	Since $\sigma^f$ is $\alpha$-stable we have $a=b=0$ and $\bar\nu(ds)=s^{-\alpha}\Gamma(1-\alpha)^{-1}\mathds{1}_{(0,\infty)}(s)\,ds$. In view of \Cref{example:classical caputo} we get 
	\begin{align*}
		\bD^f_tq(t)
		=\bD^\alpha_tq(t)
		=\frac{1}{\Gamma(1-\alpha)}\frac{d}{dt}\int_0^t(t-s)^{-\alpha}(q(s)-q(0))\,ds.
	\end{align*}
	Then the generalized fractional Kolmogorov backward equation in (\ref{problem:linear FDE1}) becomes
	\begin{align*}
	\begin{cases}
		\bD^\alpha_tq(t)=Aq(t),\quad 0<t<\infty,\\
		 q(0)=\vec{p}\in\bR^N,
	\end{cases}
	\end{align*}
	which is known to admit the explicit solution $t\mapsto E_{\alpha}(t^\alpha A)\vec{p}$, cf. \cite[p. 131]{Garrappa18}. In view of \Cref{cor:solution pol} where $\cT_tp(x)=\bE_x\left[p(X_{L_t})\right]$, for $x\in S$, we conclude
	\begin{align*}
		\overrightarrow{\cT_tp}
		=E_{\alpha}(t^\alpha A)\vec{p},\quad t>0,
	\end{align*}
	which proves the assertion. 
\end{proof}

\begin{example}[Inverse $\alpha$-stable time-changed Brownian motion]\label{ex:time-changed BM}\dbt{
	Let $X$ denote a Brownian motion on $\bR$ with generator $\cG=\frac{1}{2}\frac{d^2}{dx^2}$.
	Trivially, $X$ is polynomial and in view of \Cref{cor:moment formula for time-changed pol processes} we get
\begin{align}\label{eq:explicit moments BM}
	\bE_x\left[p(X_{L_t})\right]
	=H(x)^\top E_{\alpha}(t^\alpha \cG|_{\cP_n})\vec{p},\quad x\in \bR,
\end{align}
for all $p\in\cP_n$. Applying $\cG$ to the basis of monomials $(x^0,x^1,\dots,x^{n})$ of $\cP_n$ yields 
\begin{align*}
	G_n\coloneqq\cG|_{\cP_n}
	=\begin{pmatrix}
    0 & 0 & \binom{2}{2}=1 & 0 & \cdots & 0 & 0 \\
    0 & 0 & 0 & \binom{3}{2}=3 & 0 & \cdots & 0 \\
    \vdots & & & & \ddots & & \vdots \\
    0 & & \cdots & & 0 & \binom{n-1}{2} & 0 \\
    0 & & & \cdots & & 0 & \binom{n}{2} \\
    0 & & & \cdots & & 0 & 0 \\
    0 & & &  & \cdots &  & 0
	\end{pmatrix}.
\end{align*}
Using the power series representation of the Mittag-Leffler function, we can calculate the moments of $X_{L_t}$ as
\begin{equation}\label{eq:moment_BM}
\bE_0[X_{L_t}^n] = e_0^\top E_\alpha(t^\alpha G_n) e_n = e_0^\top \sum_{l=0}^\infty\frac{(t^\alpha G_n)^l}{\Gamma(\alpha l+1)} e_n  =\sum_{l=0}^\infty\frac{t^{\alpha l}}{\Gamma(\alpha l+1)} e_0^\top G_n^l e_n.
\end{equation}
The last factor in each summand is the top right element of the matrix power $G_n^l$. Now $G_n$ is a band matrix, with a single band of non-zero elements $g_{ij}$, described by the relation $j = i+2$. A moment of reflection reveals that also $G_n^l$ is a band matrix, with a single band of non-zero elements $g^{(l)}_{ij}$, described by $j = i + 2l$. Thus, if $n$ is odd, then $e_0^\top G_n^l e_n = 0$ for all $l \in \bN_0$, and we conclude that $\bE_0[X_{L_t}^{2k+1}] = 0$ for all $k \in \bN_0$ and $t \geq 0$. (Of course, a simple conditioning argument gives the same conclusion). If $n$ is even, we write $n = 2k$ and conclude that
\[e_0^\top G_{2k}^l e_{2k} = \begin{cases}0 \quad &\text{if} \quad k \neq l\\
\prod_{m=1}^k \binom{2m}{2} \quad &\text{if} \quad k = l\end{cases}.\]
Inserting into \eqref{eq:moment_BM}, we obtain
\begin{equation*}
\bE_0[X_{L_t}^{2k}] = \frac{t^{\alpha k}}{\Gamma(\alpha k+1)} \prod_{m=1}^k \binom{2m}{2} = \frac{t^{\alpha k} k!}{\Gamma(\alpha k+1)}(2k-1)!!,
\end{equation*}
where $(2k-1)!!$ is the double factorial, cf.~\cite[\S5.4]{NIST:DLMF}. In the boundary case $\alpha=1$, the subordinator $L$ becomes the identity, and our formula is consistent with the moments of Brownian motion without subordination.}
\end{example}
\begin{rem}
	The calculation of the matrix Mittag-Leffler function \( E_{\alpha}(A) \) for \( A=\cG|_{\cP_n} \) follows different methods based on the properties of \( A \). If \( A \) is nilpotent which translates to a generator $\cG$ that is strictly degree reducing on $\cP_l$, for $l\leq n$, the series expansion \dbt{reduces to a finite sum}, as demonstrated in \Cref{ex:time-changed BM}. If \( A \) is diagonalizable, \( E_{\alpha}(A) \) can be explicitly computed using the Jordan normal form of $A$, as shown in \Cref{rem:long range dependence}. In all other cases, numerical approximations have to be considered, see \cite{Garrappa18} for full details.
\end{rem}

\subsection{Cross-moments in equilibrium}\label{sec:cross moments}

Let $X$ be a polynomial process on $S$ with extended generator $\cG$. \dgt{For $k\in\bN$} we call a probability measure $\mu$ on $S$ \emph{$k$-limiting} for $X$ if
\begin{align*}
\lim_{t \to \infty} \bE_x[u(X_t)] = \int_S u(y) \mu(dy),
\end{align*}
for all $u \in\cP_k$ and $x\in S$. In other words, a probability measure $\mu$ is $k$-limiting, if all moments of $X$ up to order $k$ converge to the moments of $\mu$ as $t$ goes to infinity. It is obvious that a $k$-limiting measure is not unique, even when it is limiting for all $k\in \bN$, since probability measures can in general not be uniquely characterized by their moments. From \cite[Prop.~A.6]{Filipovic19}, we have the following:
\begin{itemize}
\item If $\left. \cG \right|_{\cP_{k+1}}$ is zero-stable for some $k\in \bN$, then there is an $k$-limiting measure $\mu$.
\end{itemize}

\begin{prop}[Moments in equilibrium]\label{lem:stat result polynomial processes}
	Assume $A=\cG|_{\cP_k}$ is zero-stable for fixed $k\in\{1,\dots,m\}$. Let $v$ denote the left eigenvector corresponding to the eigenvalue $0$ 
	(i.e. $v^\top A=0$), normalized such that its first element is one. Then for each $p\in\cP_k$ and all $x\in S$ we have
	\begin{align*}
		\lim_{t\to\infty}\bE_x[p(X_t)]
		=\lim_{t\to\infty}\bE_x[p(X_{L_t})]
		=v^\top\vec{p}.
	\end{align*}
	Moreover, if $\mu$ is $k$-limiting for $X$, for all $t\geq0$, then
	\begin{align}\label{eq: stationary distrib pol}
		\bE_\mu[p(X_t)]
		=\bE_\mu[p(X_{L_t})]
		=v^\top\vec{p}.
	\end{align}
\end{prop}
\begin{proof}
	Let $A=QJQ^{-1}$ denote the Jordan decomposition of A, where the top-left Jordan block of $J$ contains only the simple eigenvalue $0$. Note that all other diagonal elements of $J$ are strictly negative since $A$ is zero-stable. The columns of $Q$ represent the right (generalized) eigenvectors $A$ and the rows of $Q^{-1}$ represent the left (generalized) eigenvectors of $A$. In particular, the first column of $Q$ is a multiple of $e_0=(1,0,\dots,0)^\top$ and the first row of $Q^{-1}$ is a mulitple of $v$. Hence, 
	\begin{align*}
		\lim_{t\to\infty}e^{tA}
		=Q\lim_{t\to\infty}e^{tJ}Q^{-1}
		=Qe_0e_0^\top Q^{-1}
		=ce_0v^\top,
	\end{align*}
	for some $c\in\bR$. Now let $p\in\cP_k$ and $x\in S$. Since $X$ is $m$-polynomial, we get 
\begin{align*}
	\lim_{t\to\infty}\bE_x[p(X_t)]
	=H(x)^\top\lim_{t\to\infty}e^{tA}\vec{p}
	=cv^\top\vec{p},
\end{align*}
where $H$ is a polynomial basis $H(x)^\top=(1,h_1(x),\dots,h_{N-1}(x))$ and $N=\dim\cP_k$. With \Cref{cor:moment formula for time-changed pol processes} 
we similarly get 
\begin{align*}
	\lim_{t\to\infty}\bE_x[p(X_{L_t})]
	=H(x)^\top\lim_{t\to\infty}E_\alpha(At^\alpha)\vec{p}
	=cv^\top\vec{p}.
\end{align*}
Plugging in the constant polynomial $p\equiv 1$ forces $c=1$, and shows the first assertion.
Next, we assume $\mu$ is $k$-limiting for $X$. Then 
\begin{align*}
	\int_Sp(x)\,\mu(dx)
	=\lim_{t\to\infty}\bE_x[p(X_t)],
\end{align*}
for all $p\in\cP_k$ and $x\in S$. 
Plugging in the monomial $\vec{p}=e_i$ yields
\begin{align*}
	\int_S h_i(x)\,\mu(dx)=v^\top e_i=v_i,\quad 1\leq i\leq N-1,
\end{align*}
and hence
\begin{align*}
	\int_S H(x)\,\mu(dx)=v.
\end{align*}
Since $X$ is polynomial we get
\begin{align*}
	\bE_\mu[p(X_t)]
	=\int_S\bE_x[p(X_t)]\,\mu(dx)
	&=\left(\int_SH(x)^\top\,\mu(dx)\right)e^{tA}\vec{p}\\
	&=v^\top e^{tA}\vec{p}
	=v^\top\vec{p},	
\end{align*}
where in the last equality we used that $v^\top A=0$ for all $t\geq0$. Similarly, \Cref{cor:moment formula for time-changed pol processes} gives 
\begin{align*}
	\bE_\mu[p(X_{L_t})]
	=\int_S\bE_x[p(X_{L_t})]\,\mu(dx)
	&=\left(\int_SH(x)^\top\,\mu(dx)\right)E_{\alpha}(At^\alpha)\vec{p}\\
	&=v^\top E_{\alpha}(At^\alpha)\vec{p}
	=v^\top\vec{p}.
\end{align*}
\end{proof}
What follows is an auxiliary Lemma which allows us to determine the distribution of matrix scaled increments of $L_t$ in terms of Mittag Leffler functions.
\begin{lem}\label{prop:Laplace of increments}
	Let $L_t$, $t\geq0$, be the inverse of an $\alpha$-stable subordinator, $\alpha\in(0,1)$. Let $A$ be a zero-stable matrix, fix $s,t\geq0$, and let $F_{s,t}$ denote the cumulative distribution function (cdf) of $L_{t+s}-L_t$. Then the Laplace transform $\hat F_{s,t}$ is defined on the spectrum of $-A$ and given by 
	\begin{align}\label{eq:laplace increments}
		\hat{F}_{s,t}(-A)
		=\frac{-\alpha A(t+s)^\alpha}{\Gamma(1+\alpha)}\int_0^{\frac{t}{t+s}}\frac{E_{\alpha}(A(t+s)^\alpha(1-z)^\alpha)}{z^{1-\alpha}}\,dz+E_{\alpha}(A(t+s)^\alpha).
	\end{align}
	where $E_{\alpha}(\cdot)$, for $\alpha>0$, is the Mittag Leffler function defined in \cref{def:mittag leffler function}.
\end{lem}
\begin{proof}
	For arbitrary $\lambda>0$ we can write
	\begin{align*}
		\hat{F}_{s,t}(\lambda)
		&=\int_0^\infty e^{-\lambda x}\,F_{s,t}(dx)\\
		&=\int_0^\infty\int_0^\infty e^{-\lambda(v-u)}\,l_{t+s,t}(v,u)\,dv\,du\\
		&=\int_0^\infty\int_0^\infty e^{-\lambda|v-u|}\,l_{t+s,t}(v,u)\,dv\,du
	\end{align*} 
	where we used the fact that $L_{t+s}>L_t$ a.s. Moreover, 
	in \cite[Theorem 3.1]{Leonenko13} the authors 
	prove the identity
	\begin{align*}
		&\int_0^\infty\int_0^\infty e^{-\lambda|v-u|}\,l_{t+s,t}(v,u)\,dv\,du\\
		&\hspace{5mm}=\frac{\alpha \lambda(t+s)^\alpha}{\Gamma(1+\alpha)}\int_0^{\frac{t}{t+s}}\frac{E_{\alpha}(-\lambda(t+s)^\alpha(1-z)^\alpha)}{z^{1-\alpha}}\,dz+E_{\alpha}(-\lambda(t+s)^\alpha),
	\end{align*}
	for arbitrary $\lambda>0$.
	Since $A$ is zero-stable we can apply \Cref{lem:matrix_laplace} and plug $-A$ into $\hat{F}_{s,t}$ which shows the assertion. 	
\end{proof}
\dgt{To introduce a formula for cross-moments we recall that the multiplication map $\mull_q$, as defined in (\ref{eq:multiplication map}) for $q\in\cP_k$, admits a matrix representation $M_q$ with respect to a basis $H$ of $\cP_k$.}
\begin{thm}[Cross-moments in equilibrium]\label{lem:autocovariance}
\dbt{
Let $q \in \cP_k$, $p \in \cP_n$, and let $X$ be an $m$-polynomial process, where $k+n \le m$. Assume that $\mu$ is $m$-limiting for $X$. Denote by $H_n$ and $H_m$ two bases of $\cP_{n}$ and $\cP_m$ respectively. Furthermore, 
	assume that $\cG|_{\cP_{m}}$ is zero-stable, \dgt{and let $M_q$ denote the matrix representation of $\mull_q\colon\cP_n\to\cP_{m}$.}
	Set $A=\cG|_{\cP_k}$ and let $v$ denote the left eigenvector of $\cG|_{\cP_{m}}$ corresponding to the simple eigenvalue $0$ (i.e. $v^\top\cG|_{\cP_{m}}=0$), normalized such that its first element is one. Then
	\begin{align*}
		\bE_\mu[p(X_{t+s})q(X_t)]
		=v^\top M_qe^{sA}\vec{p} \qquad \forall\,s,t \geq 0.
	\end{align*}
	In addition, the time-changed process satisfies
	\begin{align*}
	\bE_\mu\left[p(X_{L_{t+s}})q(X_{L_t})\right] 
	=v^\top M_q\hat F_{s,t}(-A)\vec{p}
    \qquad\forall\,s,t \geq 0,
\end{align*}
where $\hat F_{s,t}$ is given by \eqref{eq:laplace increments}.}
\end{thm}
\begin{proof}
\dbt{
We use the notation of the statement and note that without loss of generality $H_m$ can be chosen such that it extends $H_n$, i.e. set   
\begin{align*}
	H_m=(h_0,\dots,h_{n-1},h_{n},\dots,h_{m}),
\end{align*}
where $H_n=(h_0,\dots,h_n)$ is an ordered basis of $\cP_{n}$. 
Now using the tower law, the conditional moment formula (\ref{eq:conditional moment formula}), and \Cref{lem:stat result polynomial processes} one gets
\begin{align*}
	\bE_\mu[p(X_{t+s})q(X_t)]
	&=\bE_\mu\left[q(X_t)\bE[p(X_{t+s})\mid X_t]\right]\\
	&=\bE_\mu\left[q(X_t)H_n(X_t)^\top e^{(t+s-t)A}\vec{p}\right]\\
	&=\bE_\mu\left[H_m(X_t)^\top M_qe^{sA}\vec{p}\right]\\
	&=v^\top M_qe^{sA}\vec{p},
\end{align*}
which shows the first statement. For the proof of the second statement, we first note that again with (\ref{eq:conditional moment formula}) and the independence of $X$ and $L$ we derive
\begin{align*}
	\bE\left[p(X_{L_{t+s}})\mid X_{L_t}, L_t\right]
		&=\bE\left[\bE\left[p(X_{L_{t+s}})\mid X_{L_t},L_t,L_{t+s}\right]\mid X_{L_t}, L_t\right]\\
		&=\bE\left[H_n(X_{L_t})^\top e^{(L_{t+s}-L_t)A}\vec{p}\mid X_{L_t}, L_t\right]\\
		&=H_n(X_{L_t})^\top \bE\left[e^{(L_{t+s}-L_t)A}\vec{p}\mid L_t\right]\\
		&=H_n(X_{L_t})^\top\left(\int_0^\infty e^{(v-L_t) A}l_{t+s}(dv\mid L_t)\right)\vec{p}\\
		&=H_n(X_{L_t})^\top \;\varpi(s\mid L_t)\;\vec{p},
\end{align*}
where $l_{t+s}(dv\mid L_t)$ is the distribution of $L_{t+s}$ conditioned on $L_t$, and
\begin{align*}
	\varpi(s\mid L_t=u)
		&=\int_0^\infty e^{(v-u) A}l_{t+s}(dv\mid L_t=u)\\
		&=l_t(u)^{-1}\int_0^\infty e^{(v-u)A}l_{t+s,t}(v,u)\,dv,\quad s,u\in(0,\infty).
\end{align*}
Using above calculations and (\ref{eq:conditional moment formula}), we get 
\begin{align*}
	\bE_\mu\left[p(X_{L_{t+s}})q(X_{L_t})\right]
	&=\bE_\mu\left[q(X_{L_t})\bE\left[p(X_{L_{t+s}})\mid X_{L_t}, L_t\right]\right]\\
	&=\bE_\mu\left[H_m(X_{L_t})^\top M_q\;\varpi(s\mid L_t)\;\vec{p}\right].
\end{align*}
Conditioning on $L_t$ and \Cref{lem:stat result polynomial processes} further yield
\begin{align*}
	&\bE_\mu\left[H_m(X_{L_t})^\top M_q\;\varpi(s\mid L_t)\;\vec{p}\right]\\
	&\hspace{5mm}=\bE_\mu\left[\bE_\mu\left[H_m(X_{L_t})^\top M_q\;\varpi(s\mid L_t)\;\vec{p}\mid L_t\right]\right]\\
	&\hspace{5mm}=\bE\left[v^\top M_q\;\varpi(s\mid L_t)\;\vec{p}\right].
\end{align*} 
Plugging in $\varpi$ gives
\begin{align*}
	&\bE\left[v^\top M_q\;\varpi(s\mid L_t)\;\vec{p}\right]\\
	&\hspace{5mm}=v^\top M_q\left(\int_0^\infty \;\varpi(s\mid L_t=u)\;l_t(u)\,du\right)\vec{p}\\\nonumber
	&\hspace{5mm}=v^\top M_q\left(\int_0^\infty\int_0^\infty
	e^{(v-u) A}l_{t+s,t}(v,u)\,dv\,du\right)\vec{p}.
\end{align*}
An appeal to \Cref{prop:Laplace of increments} shows
\begin{align*}
	&\int_0^\infty\int_0^\infty
	e^{(v-u) A}l_{t+s,t}(v,u)\,dv\,du
	=\hat{F}_{s,t}(-A)\\
	&\hspace{5mm}=\frac{-\alpha A(t+s)^\alpha}{\Gamma(1+\alpha)}\int_0^{\frac{t}{t+s}}\frac{E_{\alpha}(A(t+s)^\alpha(1-z)^\alpha)}{z^{1-\alpha}}\,dz+E_{\alpha}(A(t+s)^\alpha),
\end{align*}
where $F_{s,t}$ is the cdf of $L_{t+s}-L_t$ and which concludes the proof. 
}
\end{proof}

\begin{rem}[Long-range dependence]
\label{rem:long range dependence}
	Polynomial processes with zero-stable generator matrix exhibit short-range dependence, as their cross moments in equilibrium decay exponentially fast (see \Cref{lem:autocovariance}). However, under inverse $\alpha$-stable time change, this class of processes exhibits long-range dependence, with the aforementioned quantities decaying according to a power law with exponent $\alpha \in (0, 1)$. This property has been shown in \cite{Leonenko13} for the autocovariance and autocorrelation of time-changed Pearson diffusions (i.e., one-dimensional polynomial diffusions). Our results extend this property to the multivariate case and allow to include jumps, even with state-dependent behaviour. In more detail, fix $k\in\bN$, $\alpha\in(0,1)$, and assume all assumptions of \Cref{lem:autocovariance} are satisfied. Moreover, assume \(A =\mathcal{G}|_{\mathcal{P}_k}\) is diagonalisable, which, under mild parameter constraints, holds true for \dbt{most} examples considered in this paper. Its Jordan decomposition is given as
	\begin{align*}
		A=QDQ^{-1},
	\end{align*}
	where $D=\diag(0,-\lambda_1,\dots,-\lambda_l)$ with $\lambda_1,\dots,\lambda_l>0$ and the columns of $Q$ are the right eigenvectors of $A$. Note that in this case, for $t\geq0$, the matrix Mittag Leffler function can be decomposed as
	\begin{align*}
		E_{\alpha}(At^\alpha)
		=QE_{\alpha}(Dt^\alpha)Q^{-1}.
	\end{align*}
	where $E_{\alpha}(Dt^\alpha)=\diag(1,E_{\alpha}(-\lambda_1 t^\alpha),\dots,E_{\alpha}(-\lambda_l t^\alpha))$. Let $v$ denote the left eigenvector to the leading eigenvalue of $\cG|_{\cP_{2k}}$ and let $p,q\in\cP_k$. Then, for $s,t\geq0$, \Cref{lem:autocovariance} yields
	\begin{align*}
		\bE_\mu[p(X_{L_{t+s}})q(X_{L_t})]
		=v^\top M_qQ \hat F_{s,t}(-D)
	Q^{-1}\vec{p},
	\end{align*}
	where $\hat F_{s,t}(-D)$ is again a diagonal matrix, and given by
	\[\hat F_{s,t}(-D)=\diag(1,\hat F_{s,t}(\lambda_1),\dots,\hat F_{s,t}(\lambda_l)),\]
	with $\hat F_{s,t}$ given by \eqref{eq:laplace increments}.
	In \cite[Remark 3.3]{Leonenko13}, the authors show that for any fixed $t\geq0$,
	\begin{align*}
		\hat F_{s,t}(\lambda_i) \sim R_{s,t}(\lambda_i)\coloneqq
		\frac{1}{(t+s)^\alpha\Gamma(1-\alpha)}\left(\frac{1}{\lambda_i}+\frac{t^\alpha}{\Gamma(1+\alpha)}\right)\quad\text{as}\ s\to\infty,
	\end{align*}
	which shows that for any fixed $t\geq0$ we have 
	\begin{align}\label{eq:R}
		\bE_\mu[p(X_{L_{t+s}})q(X_{L_t})]
		\sim v^\top M_qQR_{s,t}(D)Q^{-1}\vec{p}\quad\text{as}\ s\to\infty,
	\end{align}
	where $R_{s,t}(D)=\diag(1,R_{s,t}(-\lambda_1),\dots,R_{s,t}(-\lambda_l))$. Since $R_{s,t}(-\lambda_i)\to 0$ for $s\to\infty$, we get for all $t\geq0$, and $i\in\{1,\dots,m\}$, using \Cref{lem:stat result polynomial processes}, that
    \dbt{
	\begin{align*}
		\lim_{s\to\infty}\bE_\mu[p(X_{L_{t+s}})q(X_{L_t})]
		&=v^\top M_qQ\diag(1,0,\dots,0)Q^{-1}\vec{p} = v^\top M_q e_0\;v^\top \vec{p}\\
		&=v^\top \vec{q}\; v^\top\vec{p}
		=\mu(p) \mu(q),
	\end{align*}
    }
	where $\mu(p)\coloneqq\int_S p(x)\,\mu(dx)$ and analogously for $q$. Concluding, since the right term in (\ref{eq:R}) is a linear combination of entries of $R$ we get, as $s\to\infty$,
	\begin{align*}
		\bE_\mu[p(X_{L_{t+s}})q(X_{L_t})]
		\dbt{\sim \mu(p) \mu(q)}
		+\frac{1}{(t+s)^\alpha\Gamma(1-\alpha)}\sum_{i=1}^lc_i\left(\frac{1}{\lambda_i}+\frac{t^\alpha}{\Gamma(1+\alpha)}\right),
	\end{align*}
	where $c_i\in\bR$ is a constant depending on $p$, $q$, and $Q$.	
	\end{rem}
	
\subsection{Correlation structure in equilibrium}\label{sec:correlation}
In terms of applications, \Cref{lem:autocovariance} allows to determine the correlation structure for various types of inverse $\alpha$-stable subordinated polynomial processes in equilibrium, as demonstrated in the following corollary and subsequent examples. The correlation function of a stochastic process $X$ with initial distribution $\mu$ at times $s,t \geq 0$ is defined as
\begin{align}\label{eq:correlation function}
	\corr_\mu\left(X_t,X_s\right)
		=\frac{\bE_\mu\left[X_tX_s\right]-\bE_\mu\left[X_t\right]\bE_\mu\left[X_s\right]}{\sqrt{\left(\bE_\mu[X_t^2]-\bE_\mu\left[X_t\right]^2\right)\left(\bE_\mu[X_s^2]-\bE_\mu\left[X_s\right]^2\right)}}.
\end{align}
\begin{cor}\label{cor:correlation function}
	Let $X$ be $2$-polynomial with one-dimensional state space $S \subseteq \bR$ and with $2$-limiting distribution $\mu$. Suppose that the assumptions in \Cref{lem:autocovariance} are satisfied with $k=1$, i.e. $\cG|_{\cP_1}$ has eigenvalues $0$ and $-\beta<0$. Then, for $s,t\geq0$, it holds that 
	\begin{align*}
		&\corr_\mu(X_{L_{t+s}},X_{L_t}) = \hat F_{s,t}(\beta) \\
	&\hspace{5mm}=\frac{\alpha\beta(t+s)^\alpha}{\Gamma(1+\alpha)}\int_0^{\frac{t}{t+s}}\frac{E_{\alpha}(-\beta (t+s)^\alpha(1-z)^\alpha)}{z^{1-\alpha}}\,dz+E_{\alpha}(-\beta (t+s)^\alpha).
	\end{align*} 
\end{cor}
\begin{proof}
Using the notation of the statement we have 
\begin{align}\label{eq:Jordan form of A}
		A = \cG|_{\cP_1}&=
		\begin{pmatrix}
			0 & \phi \\
			0 & -\beta 
		\end{pmatrix}
		=\begin{pmatrix}
			1&-\frac{\phi}{\beta}\\
			0&1
		\end{pmatrix}
		\begin{pmatrix}
			0&0\\
			0&-\beta
		\end{pmatrix}
		\begin{pmatrix}
			1&\frac{\phi}{\beta}\\
			0&1
		\end{pmatrix},
\end{align}
for some $\phi\in\bR$. Note that $\cG|_{\cP_1}=(\cG_{\cP_2})_{3,3}$, where $(\cG_{\cP_2})_{3,3}$ denotes the submatrix of $\cG_{\cP_2}$ with the third row and third column removed. It is then easy to verify that every left-eigenvector of $\cG_{\cP_2}$ corresponding to the eigenvalue $-\beta$ is of the form $(v_1,v_2,c)^\top$ where $c\in\bR$ and $(v_1,v_2)^\top$ is a left-eigenvector of $\cG_{\cP_1}$ corresponding to $\lambda$. 
Therefore, by \cref{eq:Jordan form of A}, we deduce that $v=(1,\frac{\phi}{\beta},c)^\top$, for some $c\in\bR$, is the left-eigenvector of $\cG|_{\cP_2}$ corresponding to the eigenvalue $0$. Now, for all $t\geq0$, \Cref{lem:stat result polynomial processes} yields
\begin{align*}
	\bE_\mu[X_{L_t}]&=v^\top(0,1,0)^\top=\frac{\phi}{\beta},\\
	\bE_\mu[X_{L_t}^2]&=v^\top(0,0,1)^\top=c.
\end{align*}
For fixed $s,t\geq0$, \cref{eq:correlation function} then gives 
\begin{align*}
	\corr_\mu(X_{L_{t+s}},X_{L_t})
	=\frac{\bE_\mu[X_{L_{t+s}}X_{L_t}]-(\frac{\phi}{\beta})^2}{c-(\frac{\phi}{\beta})^2}.
\end{align*}
In order to determine the remaining cross-moment we apply \Cref{lem:autocovariance}. 
Setting $\vec{p}=(0,1)^\top$ with respect to $H_1(x)=(1,x)$ and 
\begin{align*}
	M_p=\begin{pmatrix}
		0&1&0\\
		0&0&1
	\end{pmatrix}^\top,
\end{align*}
\Cref{lem:autocovariance} gives
\[\bE_\mu\left[X_{L_{t+s}}X_{L_t}\right] = v^\top M_p\hat F_{s,t}(-A) \vec p.\]
Now, 
\begin{align*}
\hat F_{s,t}(-A) \vec p &= 
		\begin{pmatrix}
			1&-\frac{\phi}{\beta}\\
			0&1
		\end{pmatrix}
		\begin{pmatrix}
			1&0\\
			0&\hat F_{s,t}(-\beta)
		\end{pmatrix}
		\begin{pmatrix}
			1&\frac{\phi}{\beta}\\
			0&1
		\end{pmatrix}
			\begin{pmatrix}
			0\\
			1
		\end{pmatrix} =
				\begin{pmatrix}
			\frac{\phi}{\beta} (1 - \hat F_{s,t}(\beta))\\
			\hat F_{s,t}(\beta)
		\end{pmatrix},
\end{align*}
and hence
\begin{align*}
\bE_\mu\left[X_{L_{t+s}}X_{L_t}\right] &= 				
\begin{pmatrix}
			1 &
			\frac{\phi}{\beta} &
			c
		\end{pmatrix} 
		\begin{pmatrix}
		0&0\\
		1&0\\
		0&1
	\end{pmatrix}
		\begin{pmatrix}
			\frac{\phi}{\beta} (1 - \hat F_{s,t}(\beta))\\
			\hat F_{s,t}(\beta)
		\end{pmatrix} \\
		&=  \left(c - \left(\frac{\phi}{\beta}\right)^2\right) \hat F_{s,t}(-\beta) + \left(\frac{\phi}{\beta}\right)^2.
		 \end{align*}
	Concluding, our calculations show 
	\begin{align*}
		\corr_\mu(X_{L_{t+s}},X_{L_t}) = 
 \frac{(c-(\frac{\phi}{\beta})^2)  \hat F_{s,t}(\beta) + (\frac{\phi}{\beta})^2-(\frac{\phi}{\beta})^2}{c-(\frac{\phi}{\beta})^2}= \hat F_{s,t}(\beta),
	\end{align*}
	\dbt{with $c$ cancelled out}, as desired.
\end{proof}
\begin{rem}
Our result in \Cref{cor:correlation function} can be viewed as a complement to \cite[Theorem 2.1]{Leonenko14} in the $\alpha$-stable setting. In \cite{Leonenko14}, the authors show that if the process \( X_t \) has independent increments 
	the correlation function of \( X_{L_t} \) can be explicitly computed given the identity 
\begin{align*}
&\cov(X_{L_{t+s}}, X_{L_t}) \\
&\hspace{2mm}= \var(X_1)U(t) + \mathbb{E}[X_1]^2 \int_0^t (U(t+s-\tau) + U(t-\tau)) \, U(d\tau) - U(t+s)U(t),
\end{align*}
for \( s,t \geq 0 \), where $U(t)=\frac{t^\alpha}{\Gamma(1+\alpha)}$ is the renewal function of $L_t$. While we relax the requirement on \( X_t \), allowing it to be merely $2$-polynomial, we impose the existence of a $2$-limiting distribution. 
\end{rem}
Using \Cref{cor:correlation function}, we present the correlation function of a fractional one-dimensional Pearson diffusion in equilibrium. This result recovers (and coincides with) \cite[Theorem 3.1]{Leonenko13}. Note that all one-dimensional processes considered in the subsequent examples exhibit a $2$-limiting distribution $\mu$ which is straightforward to verify due to the zero-stability of their generator matrix $\cG|_{\cP_{3}}$.
\begin{example}[Pearson Diffusions]\label{ex:pearson diffusion}
	Consider the stochastic differential equation 
	\begin{align*}
		dX_t=-\beta(X_t-\theta)\,dt+\sqrt{(a_0+a_1X_t+a_2X_t^2)}\,dW_t,
	\end{align*}
	where $\beta>0$, $a_0$, $a_1$, and $a_2$ are specified such that the square root is well-defined, and $(W_t)_{t\geq0}$ is a standard Brownian motion.
	Then, $X=(X_t)_{t\geq0}$ is called a Pearson-diffusion (see \cite{Sorensen07}), which is known to be polynomial \cite[Example 3.6]{Cuchiero12}. For every $C^2$-function $g$ its extended generator $\cG$ is given as 
	\begin{align*}
		\cG g(x)=-\beta(x-\theta)\frac{dg(x)}{dx}+\frac{1}{2}(a_0+a_1x+a_2x^2)\frac{d^2g(x)}{dx^2},\quad x\in\bR,
	\end{align*}
	and one has 
	\begin{align*}
		\cG|_{\cP_1}=
		\begin{pmatrix}
			0 & \beta\theta \\
			0 & -\beta 
		\end{pmatrix}.
	\end{align*} 
	For fixed $s,t\geq0$, a direct application of \Cref{cor:correlation function} yields 
	\begin{align*}
			\corr_\mu(X_{L_{t+s}},X_{L_t}) = \hat F_{s,t}(\beta)
\sim \frac{(t+s)^{-\alpha}}{\Gamma(1-\alpha)}\left(\frac{1}{\beta}+\frac{t^\alpha}{\Gamma(1+\alpha)}\right)\quad\text{as}\ s\to\infty,
	\end{align*}
\end{example}
\Cref{cor:correlation function} also applies to polynomial diffusions with jumps (not covered in \cite{Leonenko13}), as demonstrated in the following two examples. We note that jumps with state-dependent jump-height affect the correlation structure, in contrast to state-independent jumps.
\begin{example}[Jacobi process with jumps]
	The Jacobi process \cite{Gourieroux06} with jumps corresponding to a Poisson Process $J=(J_t)_{t\geq0}$ with intensity $\lambda$ is governed by the stochastic differential equation
\begin{align*}
    dX_t = -\beta(X_t - \theta) dt + \sigma \sqrt{X_t(1 - X_t)} dW_t+(1-2X_t)dJ_t,
\end{align*}
where $\theta \in [0,1]$ and $\beta, \sigma > 0$ and $W$ is a standard Brownian motion. It is polynomial \cite[Example 3.5]{Cuchiero12} and the size of each jump depends on the current level of the process. Specifically, when a jump takes place, the process is reflected at $1/2$, ensuring that it remains within the interval $S=[0,1]$. Its extended generator, \dbt{using the identity as truncation function}, is given by
	\begin{align*}
		\cG g(x)
		=-\beta(x-\theta)\frac{dg(x)}{dx}+\frac{1}{2}\sigma^2(x(1-x))\frac{d^2g(x)}{dx^2}+\lambda(g(1-x)-g(x)),
	\end{align*}
	where the predictable version of the corresponding jump kernel $K$ in \cite[Proposition 2.12]{Cuchiero12} is the pushforward of the L\'evy measure $\nu(d\xi)=\lambda\delta_1(d\xi)$ under the affine function $x\mapsto p_x(\xi)=-2\xi x+\xi$, for each $\xi\in\bR$, i.e. 
	\begin{align*}
		K(x,d\xi)
		=(p_x)_*\nu(d\xi)=\lambda\delta_{-2x+1}(d\xi).
	\end{align*}
	Applying $\cG$ to $(x^0,x^1)$ gives
	\begin{align*}
		\cG|_{\cP_1}
		=\begin{pmatrix}
			0&\beta\theta+\lambda\\
			0&-(\beta+2\lambda)
		\end{pmatrix}.
	\end{align*}
	Since $\beta,\lambda>0$, for fixed $s,t\geq0$, a direct application of \Cref{cor:correlation function} yields 
	\begin{align*}
		\corr_\mu(X_{L_{t+s}},X_{L_t}) = \hat F_{s,t}(2\lambda + \beta)
\sim \frac{(t+s)^{-\alpha}}{\Gamma(1-\alpha)}\left(\frac{1}{2\lambda + \beta}+\frac{t^\alpha}{\Gamma(1+\alpha)}\right),\quad\text{as}\ s\to\infty.
	\end{align*}
\end{example}
\begin{example}(L\'evy-driven Ornstein-Uhlenbeck process)
Consider for $X$ a one-dimensional stationary L\'evy-driven Ornstein Uhlenbeck process which can be characterised by the stochastic differential equation
\begin{align*}
	dX_t=-\beta(X_t-\theta)\,dt+\sigma\,dY_t,
\end{align*}
where $\beta,\sigma>0$, $Y=(Y_t)_{t\geq0}$ is a L\'evy process with characteristics $(a,b,\nu)$, and $\nu$ has finite second moments.
Then $X$ is $2$-polynomial (\cite[Example 3.3]{Cuchiero12}) and for every $C^2$-function $g$ its extended generator $\cG$ is given as
\begin{align*}
	\cG g(x)&=
	(b\sigma-\beta(x-\theta))\frac{dg(x)}{dx}
	+\frac{1}{2}(\sigma a)^2\frac{d^2g(x)}{dx^2}\\
	&\hspace{5mm}+\int_\bR\left(g(x+\sigma\xi)-g(x)-\sigma \xi\frac{dg(x)}{dx}\right)\,\nu(d\xi),
\end{align*}
\cite[Theorem 4.6.1]{Kolokoltsov11}, \dbt{using the identity as truncation function }

and 
\begin{align*}
	\cG|_{\cP_1}
	=\begin{pmatrix}
			0&\beta\theta+b\sigma\\
			0&-\beta
		\end{pmatrix}.
		\end{align*}
		Note that our choice of truncation function is feasible since $\nu$ has finite second moments. 
		For fixed $s,t\geq0$, a direct application of \Cref{cor:correlation function} yields 
	\begin{align*}
		\corr_\mu(X_{L_{t+s}},X_{L_t}) = \hat F_{s,t}(\beta)
\sim \frac{(t+s)^{-\alpha}}{\Gamma(1-\alpha)}\left(\frac{1}{\beta}+\frac{t^\alpha}{\Gamma(1+\alpha)}\right),\quad\text{as}\ s\to\infty.
	\end{align*}
\end{example}
In the following and final example, we demonstrate that \Cref{lem:stat result polynomial processes} and \Cref{lem:autocovariance} can be applied to determine the correlation structure of polynomial diffusions in a multivariate setting.
Recall that the class of polynomial diffusions is closed under polynomial transformations \cite[Section 4]{Filipovic19}.
\begin{example}(Quadratic term structure model)
Consider a quadratic term structure model for $r$ with illiquidity effects, characterized by a concatenation of diffusive periods and motionless periods of the interest rate.
 This can be specified as nonnegative quadratic function of an inverse $\alpha$-stable subordinated one-dimensional Ornstein-Uhlenbeck process $Y$, i.e.
	\begin{align*}
		r_t=R_0+R_1Y_{t}+R_2Y_{t}^2
	\end{align*}
	for appropriate $R_i\in\bR$. Here, $L_t$, $t\geq0$, is the hitting time process of the $\alpha$-stable subordinator, $\alpha\in(0,1)$, and $Y$ is given by
	\begin{align*}
		dY_t=(b-\beta Y_t)\,dt+\sigma\,dW_t,
	\end{align*}
	where $W$ is a standard Brownian motion. The joint process $X=(Y,r)$ then satisfies the dynamics 
	\begin{align*}
		\begin{pmatrix}
			dY_t\\
			dr_t
		\end{pmatrix}
		&=\left(
		\begin{pmatrix}
			b\\
			R_1b+R_2\sigma^2 + 2R_0\beta
		\end{pmatrix}
		+\begin{pmatrix}
			- \beta\\
			2R_2b + R_1\beta
		\end{pmatrix}
		Y_t
		+\begin{pmatrix}
			0\\
			- 2\beta
		\end{pmatrix}
		r_t
		\right)\,dt\\
		&\hspace{5mm}+\begin{pmatrix}
			\sigma\\
			(R_1+2R_2Y_t)\sigma
		\end{pmatrix}
		\,dW_t
	\end{align*}
	and is therefore a polynomial process \cite[Example 3.4]{Cuchiero12}. For $C^2$-functions $g$ the extended generator of $X$ is given by 
	\begin{align*}
		\cG g(x)
		&=\frac{\partial g}{\partial x_2}(x)(R_1b+R_2\sigma^2 + 2R_0\beta+(2R_2b + R_1\beta)x_1 - 2\beta x_2)\\
		&\hspace{2mm}+\frac{\partial g}{\partial x_1}(x)(b-\beta x_1)
		+\frac{1}{2}\frac{\partial^2 g}{\partial x_1^2}(x)\sigma^2\\
		&\hspace{2mm}+\frac{\partial^2 g}{\partial x_1\partial x_2}(x)(\sigma^2R_1+\sigma^22R_2x_1)\\
		&\hspace{2mm}+\frac{1}{2}\frac{\partial^2 g}{\partial x_2^2}(x)(\sigma^2R_1^2+\sigma^24R_1R_2x_1+\sigma^24R_2x_1^2),
	\end{align*}
	\cite[Proposition 2.12]{Cuchiero12} and applying $\cG$ to $(1,x_1,x_2)$
	and $(1,x_1,x_2,x_1^2,x_1x_2,x_2^2)$, respectively, gives
	\begin{align*}
		\cG|_{\cP_1}
		&=\diag\left(0,- \beta,- 2\beta\right)+N,\\
		\cG|_{\cP_2}
		&=\diag(0,-\beta,-2\beta,-2\beta,-3\beta,-4\beta)
		+\tilde N,
	\end{align*}
	where $N$ resp. $\tilde N$ is a nilpotent upper triangular matrix. If $\beta > 0$, \Cref{lem:stat result polynomial processes} and \Cref{lem:autocovariance} are applicable and the correlation structure of $(r_{L_t})_{t\geq0}$ can be determined using \cref{eq:correlation function}. With a similar approach as in the proof of \Cref{cor:correlation function}, the exact form can be computed using symbolic calculation for example using SymPy. 
\end{example}
\section{A conjecture on state-dependent subordination}
So far, we have considered stochastic processes that result from time changing a Markov process by an \emph{independent} inverse L\'evy-subordinator. This framework has been extended by \cite{Orsingher18} to a large class of semi-Markovian processes that can be constructed by \emph{dependent} subordination. \cite{Orsingher18} start by considering a `stepped' (i.e. piecewise constant) Markov process $M = (M_t)_{t \geq 0}$ on a state space $S$, which is time-changed by the first-hitting time process $(L_t)_{t \geq 0}$ of a subordinator $(\sigma_t)_{t \geq 0}$. The subordinator $\sigma$ depends on the path of $M$ and is characterised by a state-dependent L\'evy triplet $(b,a,\nu(\cdot,x))$ where $a=b=0$ and $\nu(\cdot,x)$ is a family of measures on $(0,\infty)$, indexed by $x \in S$ and with tail $s\mapsto\bar\nu(s,x)=\nu((s,\infty),x)$ satisfying $\int_0^\infty(s\wedge1)\,\nu(ds,x)<\infty$ for each $x\in S$. Subsequently, \cite{Orsingher18} use weak limits to generalize this construction to non-stepped semi-Markov processes. Moreover, they show in \cite[p. 832]{Orsingher18} that functionals $q(t,x) = \bE_x[f(Y_t)]$ of the time-changed process $Y_t = M_{L_t}$ satisfy a Volterra integral-differential equation of the form 
	\begin{align}\label{eq:gen_reg_Rieman_Louiv}
				\begin{cases}
					\frac{d}{dt}\int_0^tq(s,\cdot)\,\bar\nu(t-s,\cdot)\,ds-\bar\nu(t,\cdot)\,q(0,\cdot)=(\cG q)(t, \cdot),\\
					q(0)=f\in\cB(S).
				\end{cases}
	\end{align}
	
In this context, we put forward the following conjecture: 
\begin{quote}
There exists a semi-Markovian process $Y$ on a state-space $S \subseteq \bR^n$, resulting from a non-trivial state-dependent time change as described in \cite{Orsingher18}, with the polynomial property, i.e. for every $m \in \bN$ and polynomial $u \in \cP_m$ we have that 
\[q(t,.): x \mapsto \bE_x[u(Y_t)] \quad \text{is again in $\cP_m$.}\]
\end{quote}
While we were not able to \dbt{rigorously} construct such a process, our conjecture is based on the following observation: There is a Markov process $X$ on $S = [0,\infty)$ and a state-dependent L\'evy measure $\nu(ds,x)$, such that the corresponding Kolmogorov equation \eqref{eq:gen_reg_Rieman_Louiv} has polynomial solutions, i.e. when $q(0) = u \in \cP_m$ then also $q(t,.) \in \cP_m$ for all $t \geq 0$. In concrete terms, we show the following: 
	
\begin{prop}\label{prop:state dependent}
	Let $X$ be a solution of the SDE
		\begin{align*}
	dX_t
	= b\,dt+\sigma\sqrt{X_t}\,dW_t,\quad b,\sigma\in\bR_+, X_0 \in S = (0,\infty),
\end{align*}
 and denote its extended generator by $\cG$. Moreover, let
	\begin{align}\label{eq:kappa}
	\bar\nu(t,x)=\frac{\kappa(t)}{x},\quad t>0,\ x\in S,
\end{align}
where $\kappa$ is a non-increasing function with $\lim_{t \downarrow 0}\kappa(t) = \infty$ and $\int_0^1\kappa(s) ds < \infty$. Then, for any polynomial initial condition $u \in \cP_m$, there exists a polynomial solution $q(t,x)$ of the form 
\begin{equation}\label{eq:ansatz}
q(t,x) = c_0 + \sum_{j=1}^m x^j c_j(t) 
\end{equation}
to the Volterra integral-differential equation \eqref{eq:gen_reg_Rieman_Louiv}.
\end{prop} 
\begin{rem}
\begin{enumerate}[(a)]
\item We note that choosing $\bar \nu(t,x)$ as in \eqref{eq:kappa} immediately yields a family of state-dependent measures $\nu(dt,x)$ by setting 
\begin{align*}
	\nu((t,\infty),x)=\bar\nu(t,x) = \frac{\kappa(t)}{x},\quad t\ge0, x \in S.
\end{align*}
Since $\kappa$ is non-increasing, $\nu(\cdot,x)$ is non-negative. Moreover, $\nu$ satisfies $\nu(0,\infty) = \infty$ and 
\begin{align*}
		\int_0^\infty(s\wedge1)\,\nu(ds,x)
		&=\int_0^1s\,\nu(ds,x)+\bar\nu(1,x)\\
		&=\int_0^1\int_0^s\,dt\,\nu(ds,x)+\bar\nu(1,x)\\
		&=\int_0^1\int_t^1\,\nu(ds,x)\,dt+\bar\nu(1,x)\\
		&=\int_0^1\bar\nu(t,x)\,dt = \frac{1}{x}\int_0^1\kappa(t)\,dt <\infty,
	\end{align*}
	for all $x\in S$ which shows that also \eqref{con:levy measure integrability} is satisfied. 
\item While \cite{Orsingher18} also provide some existence results for processes with dependent subordination, they assume that 
\begin{align*}
	f(\lambda,x)
	=\int_0^\infty(1-e^{-\lambda s})\,\nu(ds,x),\quad \lambda>0,
\end{align*}
is bounded over $x \in S$, which is not the case in our conjectured example. 
\end{enumerate}
\end{rem}

\begin{proof}
The extended generator $\cG$ of X is given, for $g \in C^2(0,\infty)$ by  
\begin{align*}
	\cG g(x)
	=b\frac{dg(x)}{dx}+\frac{1}{2}\sigma^2x\frac{d^2g(x)}{dx^2},
\end{align*}
see \cite[Section 4]{Cuchiero12}. Representing $\cG$ in the basis $H(x) = (x^0,x^1,\dots,x^m)$ yields the $(m+1)\times(m+1)$ matrix 
\begin{align*}
    	A_m = \cG|_{\cP_m}=\begin{pmatrix}
    		0 & b & 0 & \dots  \\
    		0 & 0 & 2b+\sigma^2 & 0 & \dots \\
    		0 & 0 & 0 & 3b+3\sigma^2&0&\dots \\
    		&&&&\ddots&\\
    		0&\dots&&& 0 &mb+\frac{m(m-1)}{2}\sigma^2\\
    		0&\dots&&&& 0
    	\end{pmatrix}.
\end{align*}
Applying $A_m$ to the Ansatz \eqref{eq:ansatz}, we obtain 
\begin{equation}\label{eq:RHS}
\overrightarrow{\cG q(t,x)} = \left(bc_1(t), (2b + \sigma^2) c_2(t), \dotsc, \left(mb+\frac{m(m-1)}{2}\sigma^2\right) c_m(t), 0\right)^\top,
\end{equation}
as vector representation in the basis $H(x)$ of the right hand side of \eqref{eq:gen_reg_Rieman_Louiv}. Switching to the left hand side of \eqref{eq:gen_reg_Rieman_Louiv}, we set 
\begin{align*}
	L(t,x)=\frac{d}{dt}\int_0^tq(s,x)\,\bar\nu(t-s,x)\,ds-\bar\nu(t,x)\,q(0,x),\quad t\geq0.
\end{align*}
 Plugging in \eqref{eq:kappa} and with the ansatz \eqref{eq:ansatz}, we obtain
\begin{align}
	\nonumber 
	L(t,x)&=\frac{d}{dt}\int_0^t\left(c_0 + \sum_{j=1}^m x^jc_j(s)\right)\kappa(t-s)x^{-1}\,ds\\\nonumber
	&\hspace{1cm}-\left(c_0 + \sum_{j=1}^m x^jc_j(0)\right)\kappa(t)x^{-1}\\\nonumber
	&=x^{-1}c_0 \left(\frac{d}{dt}\int_0^t\kappa(t-s)\,ds-\kappa(t)\right)\\\nonumber
	&\hspace{1cm}+\sum_{j=0}^{m-1}x^j\left(\frac{d}{dt}\int_0^tc_{j+1}(s)\kappa(t-s)\,ds-c_{j+1}(0)\kappa(t)\right)\\\nonumber
    &=\sum_{j=0}^{m-1}x^j\left(\frac{d}{dt}\int_0^tc_{j+1}(s)\kappa(t-s)\,ds-c_{j+1}(0)\kappa(t)\right).
\end{align}
This is a polynomial of degree $m-1$, and we can express it in the basis $H(x)$ as vector $\overrightarrow{L(t,x)}$ with components 
\[\overrightarrow{L(t,x)}_j = \frac{d}{dt}\int_0^tc_{j+1}(s)\kappa(t-s)\,ds-c_{j+1}(0)\kappa(t), \quad j=0,\dotsc,m-1,\]
and $\overrightarrow{L(t,x)}_m = 0$. Setting $\overrightarrow{L(t,x)}$ equal to \eqref{eq:RHS} we obtain the scalar generalized fractional linear differential equations
\[\frac{d}{dt}\int_0^tc_{j+1}(s)\kappa(t-s)\,ds-c_{j+1}(0)\kappa(t) = \left(jb + \frac{j(j-1)}{2}\sigma^2\right) c_{j+1},\]
which, together with the initial conditions for $c_j(0)$, can be solved to obtain $c_{j+1}(t), j=0, \dotsc, m-1$. The coefficient $c_0$ is constant and hence directly determined from the initial condition. Together, we have shown that there is a polynomial solution $q(t,x)$ of form \eqref{eq:ansatz} to the Volterra integral-differential equation \eqref{eq:gen_reg_Rieman_Louiv}, as claimed.
\end{proof}

\printbibliography

@article{Cuchiero24,
	author = {Cuchiero, Christa and Di Persio, Luca and Guida, Francesco and Svaluto-Ferro, Sara},
	date = {2024/09/01/},
	doi = {https://doi.org/10.1016/j.spa.2024.104392},
	isbn = {0304-4149},
	journal = {Stochastic Processes and their Applications},
	pages = {104392},
	title = {Measure-valued affine and polynomial diffusions},
	url = {https://www.sciencedirect.com/science/article/pii/S030441492400098X},
	volume = {175},
	year = {2024}}

@article{Leonenko14,
	author = {Leonenko, Nikolai and Meerschaert, Mark and Schilling, Ren{\'e} and Sikorskii, Alla},
	doi = {10.1685/journal.caim.483},
	journal = {Communications in Applied and Industrial Mathematics},
	month = {10},
	title = {Correlation Structure of Time-Changed L{\'e}vy Processes},
	volume = {6},
	year = {2014}}

@book{Widder46,
	author = {David V. Widder},
	publisher = {Princeton Mathematical Series},
	title = {The {L}aplace Transform},
	year = {1946}}

@book{Perko13,
	author = {Perko, Lawrence},
	publisher = {Springer Science \& Business Media},
	title = {Differential equations and dynamical systems},
	volume = {7},
	year = {2013}}

@book{Higham08,
	author = {Higham, NJ},
	publisher = {SIAM},
	title = {Functions of Matrices: Theory and Computation},
	year = {2008}}

@misc{Guida21,
	archiveprefix = {arXiv},
	author = {Christa Cuchiero and Francesco Guida and Luca di Persio and Sara Svaluto-Ferro},
	eprint = {2112.15129},
	primaryclass = {math.PR},
	title = {Measure-valued affine and polynomial diffusions},
	url = {https://arxiv.org/abs/2112.15129},
	year = {2021}}

@article{Bingham71,
	author = {Bingham, N. H.},
	date = {1971/03/01},
	doi = {10.1007/BF00538470},
	id = {Bingham1971},
	isbn = {1432-2064},
	journal = {Zeitschrift f{\"u}r Wahrscheinlichkeitstheorie und Verwandte Gebiete},
	number = {1},
	pages = {1--22},
	title = {Limit theorems for occupation times of Markov processes},
	url = {https://doi.org/10.1007/BF00538470},
	volume = {17},
	year = {1971}}

@misc{Nane11,
	archiveprefix = {arXiv},
	author = {Mark M. Meerschaert and Erkan Nane and P. Vellaisamy},
	eprint = {1007.5051},
	title = {The fractional Poisson process and the inverse stable subordinator},
	year = {2011}}

@book{Kolokoltsov11,
	address = {Berlin, New York},
	author = {Vassili N. Kolokoltsov},
	doi = {doi:10.1515/9783110250114},
	isbn = {9783110250114},
	lastchecked = {2024-06-11},
	publisher = {De Gruyter},
	title = {Markov Processes, Semigroups and Generators},
	url = {https://doi.org/10.1515/9783110250114},
	year = {2011}}

@techreport{Sorensen07,
	author = {Michael Sorensen and Julie Lyng Forman},
	institution = {Department of Economics and Business Economics, Aarhus University},
	month = Sep,
	number = {2007-28},
	title = {{The Pearson diffusions: A class of statistically tractable diffusion processes}},
	type = {CREATES Research Papers},
	url = {https://ideas.repec.org/p/aah/create/2007-28.html},
	year = 2007}

@book{Jacod03,
	author = {Jacob, J. and Shiryayev, A. N.},
	publisher = {Springer},
	title = {Limit Theorems for Stochastic Processes},
	year = 2002}

@article{Gourieroux06,
	author = {Christian Gourieroux and Joann Jasiak},
	doi = {https://doi.org/10.1016/j.jeconom.2005.01.014},
	issn = {0304-4076},
	journal = {Journal of Econometrics},
	number = {1},
	pages = {475-505},
	title = {Multivariate Jacobi process with application to smooth transitions},
	url = {https://www.sciencedirect.com/science/article/pii/S0304407605000199},
	volume = {131},
	year = {2006}}

@article{Engel01,
	author = {Engel, Klaus-Jochen and Nagel, Rainer},
	date = {2001/06/01},
	doi = {10.1007/s002330010042},
	id = {Engel2001},
	isbn = {1432-2137},
	journal = {Semigroup Forum},
	number = {2},
	pages = {278--280},
	title = {One-parameter semigroups for linear evolution equations},
	url = {https://doi.org/10.1007/s002330010042},
	volume = {63},
	year = {2001}}

@article{Cuchiero21,
	author = {Cuchiero, Christa and Svaluto-Ferro, Sara},
	date = {2021/04/01},
	doi = {10.1007/s00780-021-00450-x},
	id = {Cuchiero2021},
	isbn = {1432-1122},
	journal = {Finance and Stochastics},
	number = {2},
	pages = {383--426},
	title = {Infinite-dimensional polynomial processes},
	url = {https://doi.org/10.1007/s00780-021-00450-x},
	volume = {25},
	year = {2021}}

@article{Kobayashi2011,
	author = {Kobayashi, Kei},
	journal = {Journal of Theoretical Probability},
	number = {3},
	pages = {789--820},
	publisher = {Springer},
	title = {Stochastic calculus for a time-changed semimartingale and the associated stochastic differential equations},
	volume = {24},
	year = {2011}}

@misc{NIST:DLMF,
	howpublished = {\url{https://dlmf.nist.gov/}, Release 1.2.4 of 2025-03-15},
	key = {{\relax DLMF}},
	note = {F.~W.~J. Olver, A.~B. {Olde Daalhuis}, D.~W. Lozier, B.~I. Schneider, R.~F. Boisvert, C.~W. Clark, B.~R. Miller, B.~V. Saunders, H.~S. Cohl, and M.~A. McClain, eds.},
	title = {{\it NIST Digital Library of Mathematical Functions}},
	url = {https://dlmf.nist.gov/}}

@article{Chen17,
	author = {Chen, Zhen-Qing},
	date = {2017/09/01/},
	doi = {https://doi.org/10.1016/j.chaos.2017.04.029},
	isbn = {0960-0779},
	journal = {Chaos, Solitons \& Fractals},
	journal1 = {Future Directions in Fractional Calculus Research and Applications},
	pages = {168--174},
	title = {Time fractional equations and probabilistic representation},
	url = {https://www.sciencedirect.com/science/article/pii/S0960077917301649},
	volume = {102},
	year = {2017}}

@book{Gripenberg90,
	address = {Cambridge},
	author = {Gripenberg, G. and Londen, S. O. and Staffans, O.},
	db = {Cambridge Core},
	doi = {DOI: 10.1017/CBO9780511662805},
	dp = {Cambridge University Press},
	isbn = {9780521372893},
	publisher = {Cambridge University Press},
	series = {Encyclopedia of Mathematics and its Applications},
	title = {Volterra Integral and Functional Equations},
	url = {https://www.cambridge.org/core/product/2A26C67AFCE028453C046D0DBFDE418C},
	year = {1990}}

@article{Filipovic19,
	author = {Damir Filipovi{\'c} and Martin Larsson and Sergio Pulido},
	doi = {https://doi.org/10.1016/j.spa.2019.06.010},
	issn = {0304-4149},
	journal = {Stochastic Processes and their Applications},
	number = {4},
	pages = {1947-1971},
	title = {Markov cubature rules for polynomial processes},
	url = {https://www.sciencedirect.com/science/article/pii/S0304414919303850},
	volume = {130},
	year = {2020}}

@article{Filipovic16,
	author = {Filipovi{\'c}, Damir and Larsson, Martin},
	journal = {Finance and Stochastics},
	number = {4},
	pages = {931--972},
	publisher = {Springer},
	title = {Polynomial diffusions and applications in finance},
	volume = {20},
	year = {2016}}

@article{Leonenko13,
	author = {Nikolai N. Leonenko and Mark M. Meerschaert and Alla Sikorskii},
	doi = {https://doi.org/10.1016/j.camwa.2013.01.009},
	issn = {0898-1221},
	journal = {Computers \& Mathematics with Applications},
	note = {Fractional Differentiation and its Applications},
	number = {5},
	pages = {737-745},
	title = {Correlation structure of fractional Pearson diffusions},
	url = {https://www.sciencedirect.com/science/article/pii/S0898122113000266},
	volume = {66},
	year = {2013}}

@inbook{Cox15,
	address = {Cham},
	author = {Cox, David A. and Little, John and O'Shea, Donal},
	booktitle = {Ideals, Varieties, and Algorithms: An Introduction to Computational Algebraic Geometry and Commutative Algebra},
	doi = {10.1007/978-3-319-16721-3_2},
	isbn = {978-3-319-16721-3},
	pages = {49--119},
	publisher = {Springer International Publishing},
	title = {Gr{\"o}bner Bases},
	url = {https://doi.org/10.1007/978-3-319-16721-3_2},
	year = {2015}}

@book{Feller71,
	address = {New York},
	author = {Feller, William},
	description = {Feller volume 2},
	mrclass = {60.00},
	mrnumber = {MR0270403 (42 \#5292)},
	pages = {xxiv+669},
	publisher = {John Wiley \& Sons Inc.},
	series = {Second edition},
	title = {An introduction to probability theory and its applications. {V}ol. {II}.},
	year = 1971}

@techreport{Filipovic17,
	author = {Damir Filipovi{\'c} and Martin Larsson},
	institution = {Swiss Finance Institute},
	month = Nov,
	number = {17-60},
	title = {{Polynomial Jump-Diffusion Models}},
	type = {Swiss Finance Institute Research Paper Series},
	url = {https://ideas.repec.org/p/chf/rpseri/rp1760.html},
	year = 2017}

@inbook{Bertoin99,
	address = {Berlin, Heidelberg},
	author = {Bertoin, Jean},
	booktitle = {Lectures on Probability Theory and Statistics: Ecole d'Et{\'e} de Probailit{\'e}s de Saint-Flour XXVII - 1997},
	doi = {10.1007/978-3-540-48115-7_1},
	editor = {Bernard, Pierre},
	isbn = {978-3-540-48115-7},
	pages = {1--91},
	publisher = {Springer Berlin Heidelberg},
	title = {Subordinators: Examples and Applications},
	url = {https://doi.org/10.1007/978-3-540-48115-7_1},
	year = {1999}}

@book{Schilling17,
	address = {Cambridge},
	author = {Schilling, Ren{\'e}L.},
	edition = {Second edition},
	isbn = {9781316620243; 1316620247},
	la = {eng},
	lk = {https://worldcat.org/title/989971920},
	publisher = {Cambridge University Press},
	title = {Measures, integrals and martingales},
	year = {2017}}

@book{Sikorskii12,
	address = {Berlin, Boston},
	author = {Mark M. Meerschaert and Alla Sikorskii},
	doi = {doi:10.1515/9783110258165},
	isbn = {9783110258165},
	lastchecked = {2023-06-02},
	publisher = {De Gruyter},
	title = {Stochastic Models for Fractional Calculus},
	url = {https://doi.org/10.1515/9783110258165},
	year = {2012}}

@book{Kilbas06,
	address = {Amsterdam},
	author = {Kilbas, A. A. and Srivastava, H. M. and Trujillo, Juan J.},
	edition = {1st ed},
	isbn = {9780444518323; 0444518320; 0080462073; 9780080462073},
	la = {eng},
	lk = {https://worldcat.org/title/162586541},
	publisher = {Elsevier},
	series = {North-Holland mathematics studies},
	title = {Theory and applications of fractional differential equations},
	year = {2006}}

@book{Schilling10,
	address = {Berlin, New York},
	author = {Ren{\'e} L. Schilling and Renming Song and Zoran Vondracek},
	doi = {doi:10.1515/9783110215311},
	isbn = {9783110215311},
	lastchecked = {2023-05-25},
	publisher = {De Gruyter},
	title = {Bernstein functions},
	url = {https://doi.org/10.1515/9783110215311},
	year = {2010}}

@article{Garrappa18,
	author = {Garrappa, Roberto and Popolizio, Marina},
	date = {2018/10/01},
	doi = {10.1007/s10915-018-0699-5},
	id = {Garrappa2018},
	isbn = {1573-7691},
	journal = {Journal of Scientific Computing},
	number = {1},
	pages = {129--153},
	title = {Computing the Matrix Mittag-Leffler Function with Applications to Fractional Calculus},
	url = {https://doi.org/10.1007/s10915-018-0699-5},
	volume = {77},
	year = {2018}}

@article{Cuchiero12,
	author = {Cuchiero, Christa and Keller-Ressel, Martin and Teichmann, Josef},
	date = {2012/10/01},
	doi = {10.1007/s00780-012-0188-x},
	id = {Cuchiero2012},
	isbn = {1432-1122},
	journal = {Finance and Stochastics},
	number = {4},
	pages = {711--740},
	title = {Polynomial processes and their applications to mathematical finance},
	url = {https://doi.org/10.1007/s00780-012-0188-x},
	volume = {16},
	year = {2012}}

@article{Toaldo15,
	author = {Toaldo, Bruno},
	date = {2015/01/01},
	doi = {10.1007/s11118-014-9426-5},
	id = {Toaldo2015},
	isbn = {1572-929X},
	journal = {Potential Analysis},
	number = {1},
	pages = {115--140},
	title = {Convolution-Type Derivatives, Hitting-Times of Subordinators and Time-Changed C0-semigroups},
	url = {https://doi.org/10.1007/s11118-014-9426-5},
	volume = {42},
	year = {2015}}

@article{Orsingher18,
	author = {Enzo Orsingher and Costantino Ricciuti and Bruno Toaldo},
	doi = {https://doi.org/10.1016/j.jfa.2018.02.011},
	issn = {0022-1236},
	journal = {Journal of Functional Analysis},
	number = {4},
	pages = {830-868},
	title = {On semi-Markov processes and their Kolmogorov's integro-differential equations},
	url = {https://www.sciencedirect.com/science/article/pii/S002212361830079X},
	volume = {275},
	year = {2018}}

@article{Meerschaert13,
	author = {{Meerschaert, M. M.} and {Straka, P.}},
	doi = {10.1051/mmnp/20138201},
	journal = {Math. Model. Nat. Phenom.},
	number = 2,
	pages = {1-16},
	title = {Inverse Stable Subordinators},
	url = {https://doi.org/10.1051/mmnp/20138201},
	volume = 8,
	year = 2013}

@article{Meerschaert04,
	author = {Meerschaert, Mark and Scheffler, Hans-Peter},
	doi = {10.1017/S002190020002043X},
	journal = {Journal of Applied Probability},
	month = {09},
	pages = {623-638},
	title = {Limit theorems for continuous-time random walks with infinite mean waiting times},
	volume = {41},
	year = {2004}}

\end{document}